\documentclass[a4paper,12pt,reqno]{amsart}
\usepackage{float}
\usepackage{graphicx}
\usepackage{amsmath,amsthm,amssymb,amsfonts}
\usepackage{subcaption}
\usepackage{latexsym}
\usepackage{color}     
\usepackage{hyperref}
\newcommand{\len}{\operatorname{length}}
\newcommand{\diam}{\operatorname{diam}}
\newcommand{\cE}{\mathcal{E}}
\newcommand{\sign}{\text{sign}}

\usepackage[T1]{fontenc}
\usepackage{amsmath}                       
\usepackage{amssymb}                       
\usepackage{amsfonts}                      
\usepackage{amsthm}     
\usepackage{enumerate}

\newcommand{\RR}{\mathbb R}
\newcommand{\NN}{\mathbb N}

\newcommand{\cB}{\mathcal{B}}

\newcommand{\cH}{\mathcal{H}}

\newcommand{\eps}{\varepsilon}

\newcommand{\cof}{\operatorname{cof}}

\newcommand{\mbf}[1]{\boldsymbol{#1}}
\newcommand{\To}{\longrightarrow}

\newcommand{\abs}[1]{\left\vert#1\right\vert}

\newcommand{\absn}[1]{\vert#1\vert}
\newcommand{\absb}[1]{\big\vert#1\big\vert}

\newcommand{\norm}[1]{\left\|#1\right\|}

\newcommand{\dist}[2] {\operatorname{dist}\left(#1;#2\right)}

\newcommand{\mysetr}[2] {\left\{#1\,\left|\,#2\right.\right\}}
\newcommand{\mysetl}[2] {\left\{\left.#1\,\right|\,#2\right\}}

\newcommand{\be}{\begin{equation}}
\newcommand{\ee}{\end{equation}}
\newcommand{\bald}{\begin{aligned}}
\newcommand{\eald}{\end{aligned}}

\newtheorem{theorem}{Theorem}

\newtheorem{thm}[theorem]{Theorem}
\newtheorem{cor}[theorem]{Corollary}
\newtheorem{lem}[theorem]{Lemma}
\newtheorem{prop}[theorem]{Proposition}
\theoremstyle{definition}
\theoremstyle{remark}
\newtheorem{rem}[theorem]{Remark}
\renewcommand{\proofname}{\bfseries{Proof}}
\renewenvironment{proof}[1][\proofname]{\par
  \normalfont
  \trivlist
  \item[\hskip\labelsep\itshape
    \bfseries{#1.}]\ignorespaces
}{%
  \qed\endtrivlist
}

\numberwithin{equation}{section}

\begin{document}

\title[Global injectivity and approximation]{Global injectivity in second-gradient Nonlinear Elasticity and its approximation with penalty terms}
\
\author{Stefan Kr\"omer} \address{Stefan Kr\"omer, The Czech Academy of Sciences, Institute of Information Theory and Automation, Pod vod\'{a}renskou v\v{e}\v{z}\'{\i}~4, 182~08~Praha~8, Czech Republic (corresponding author), \email{}{skroemer@utia.cas.cz}}
\author{Jan Valdman} \address{Jan Valdman, The Czech Academy of Sciences, Institute of Information Theory and Automation, Pod vod\'{a}renskou v\v{e}\v{z}\'{\i}~4, 182~08~Praha~8, Czech Republic, \email{}{valdman@utia.cas.cz}}
\date{\today}

\begin{abstract}
We present a new penalty term approximating the Ciarlet-Ne\v{c}as condition (global invertibility of deformations) as a soft constraint for hyperelastic materials.
For non-simple materials including a suitable higher order term in the elastic energy, we prove that 
the penalized functionals converge to the original functional subject to the Ciarlet-Ne\v{c}as condition.
Moreover, the penalization can be chosen in such a way that for all low energy deformations, self-interpenetration is completely avoided already at all sufficiently small finite values of the penalization parameter.
We also present numerical experiments in 2d illustrating our theoretical results.
\end{abstract}

\maketitle


\maketitle


\section{Introduction\label{sec:intro}}

Nonlinear elasticity models the behavior of a solid body subject to relatively strong external forces causing large deformations, albeit not large enough to cause irreversible damage. For a general introduction to the topic, we refer to \cite{Sil97B,Cia88B,Ba83a}.
It is clear that the deformation of such a body, the map $y:\Omega\to \RR^d$ mapping
the ``reference configuration'' $\Omega\subset \RR^d$ to a deformed state. Here, typically $d=3$ and $\Omega$ is the domain occupied by the elastic body in its stress-free state before external forces are applied. Clearly, any realistic deformation should always be injective, i.e., the body should not interpenetrate itself in any way. 
\begin{figure}
        \centering
        \includegraphics[width=0.6\textwidth]{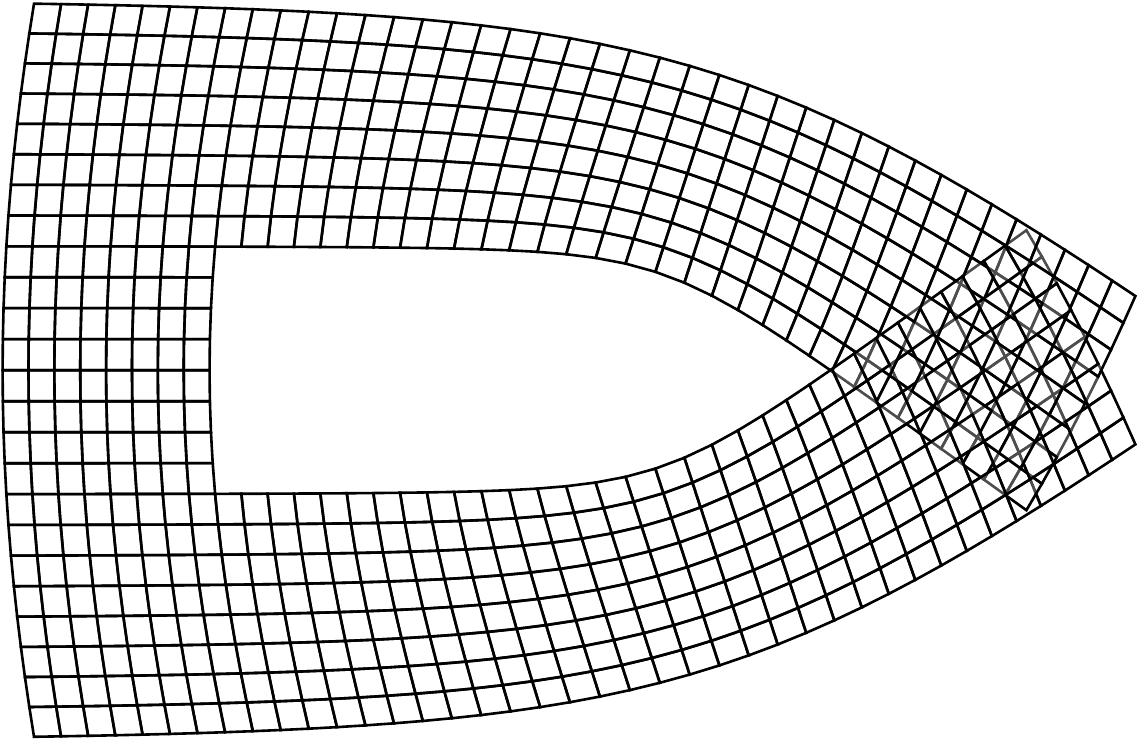}		
    \caption{Example of deformed domain with overlapping ends.}
		\label{fig:pincers_geometry_deformed_stronger}
\end{figure}
Here, we focus on the framework of hyperelasticity, that is, models where the body does not dissipate energy when deformed, instead simply storing in an internal elastic energy whose mathematical description as a functional depending on $y$ fully determines its response to external forces. 
The existence of energy minimizers in this framework was pioneered by John Ball \cite{Ba77a}. 
This theory allows us to enforce a weak form of local injectivity of the deformation, i.e., $\det \nabla y>0$ a.e.~in $\Omega$ for all deformation with finite internal energy, by using an elastic energy density which becomes infinite as $\det \nabla y\to 0^+$. However, having $\det \nabla y>0$ a.e.~is still not enough to ensure local injectivity everywhere, and globally, a loss of injectivity by, say, two different ends of the body overlapping each other (see Figure
\ref{fig:pincers_geometry_deformed_stronger}), is not automatically prevented \cite{Ba81a}. For instance, in addition to the positive determinant, to derive local (and global) injectivity, the result of \cite{Ba81a} requires global injectivity of $y$ on $\partial\Omega$ as a prerequisite, a feature hard to prove and difficult to handle as a constraint (unless it is already given in form of a Dirichlet condition). 
On the other hand, the positive determinant implies local injectivity in a neigborhood of almost every point in $\Omega$ given quite natural assumptions on the regularity of the deformation \cite{FoGa95a}. 
A direct approach to local injectivity \emph{everywhere} 
was provided in \cite{HeaKroe09a}, where a uniform positive lower bound on $\det \nabla u$ is derived, but this only works 
in the framework of so-called non-simple materials, where the internal elastic energy is assumed to contain a suitable term involving second order derivatives $\nabla^2 u$ (or higher order, and the coercivity conditions in \cite{HeaKroe09a} automatically entail $y\in C^1$ by embedding), as opposed to classical hypereleasticity for which the energy only depends on $\nabla y$. 
For further information on the topic of locally injective deformations, we also refer to \cite{Ba02a,FoGa95B}. In addition, there is literature available discussing settings that allow for cavitation (see, e.g., \cite{HeMo15a} and the references therein), but we will rule that out by assumption in this article (coercivity in $W^{1,p}$ with $p>d$).
The standard constraint nowadays used to ensure global injectivity of the deformation is the Ciarlet-Ne\v{c}as condition \cite{CiaNe87a}:
\begin{align}\label{ciarletnecas}
\int_\Omega \det(\nabla y)\,dx = \abs{y(\Omega)}.
\end{align}
The numerical treatment of this condition is not well understood, however. In particular, to our knowledge there is so far no example where \eqref{ciarletnecas} is enforced as a hard constraint on the discrete level in a numerical scheme with provable convergence.

Artificially including terms with derivatives of second or higher order in the energy, possibly multiplied with a small parameter, can also serve as regularization, and among other thing, the results of \cite{HeaKroe09a} can then be used to obtain local injectivity.
However, with such a regularization, there is a risk of a Lavrentiev phenomenon occuring, that is, the minimal energy of the regularized problem might remain strictly above the minimal energy for the original problem without the higher order term, even if the regularization parameter converges to zero. In particular, it is known that the infimum of the internal energy in spaces with higher integrability of $\nabla y$ may be too large \cite{FoHruMi03a}. Similar issues might occur when discretizing, as the typical finite element spaces are all subsets of $W^{1,\infty}$. For simple materials, i.e., models without higher order terms, a way out was shown in \cite{Ne90a}, using an artificially introduced auxiliary field, and
similarly in \cite{MieRou16a}\footnote{also for more complicated models including plasticity}. For the variant with a higher order term also treated in \cite{MieRou16a},
a Lavrentiev phenomenon cannot be ruled out.

In this article, we further elaborate on the approach of 
\cite{MieRou16a} in presence of a higher order term
\begin{align}\label{secondgradient0}
	\sigma \int_\Omega \abs{D^2 y}^s \,dx\quad \text{($\sigma>0$ is a fixed parameter)}
\end{align}
using soft constraints, i.e., everywhere finite terms in the energy depending on control parameters $\eps_i$, converging to the singular determinant term and the Ciarlet-Ne\v{c}as condition \eqref{ciarletnecas}, respectively, as $\eps_i\to 0$, $i=1,2$. In particular, 
the penalization term for the latter in \cite{MieRou16a} is defined as
\begin{align}\label{ciarletnecassoft}
	\frac{1}{\eps_2} \left(\int_\Omega \det(\nabla y)\,dx -\abs{y(\Omega)}\right).
\end{align}
After a discretizing with mesh size $h$, it is shown in \cite{MieRou16a} that 
the energy minima converge 
for these approximations (the energies even $\Gamma$-converge)
as $(\eps_1,\eps_2,h)\to (0,0,0)$ in the scaling regime $\frac{h}{\eps_1}\to 0$. 
They also show that there is also convergence even for $\sigma=0$, but then
the suitable scaling regime, while shown to exists, is not explicitly known and therefore effectively impossible to exploit in practice.

How to actually compute \eqref{ciarletnecassoft} is not explained in \cite{MieRou16a}, however, and since this term is nonlocal and, in particular, not a standard integral functional, this can in fact be quite problematic to implement. This is further aggravated by the presence of a higher order term like \eqref{secondgradient0} which rules out the (direct) use of piecewise affine finite elements. An alternative, more accessible penalization was recently proposed and studied in \cite{BaRei18aP}, but only for beams (effectively 1d).
Here, after precisely outlining the model we work with in Section~\ref{sec:model}, 
we show that instead of 
\eqref{ciarletnecassoft}, alternatives in the form of double integrals can be used (Section~\ref{sec:cn}), for example 
\begin{align}\label{sCN1ex}
\bald
	&\frac{1}{\eps_2}
 \int_{(\Omega\times \Omega)} 
	\frac{1}{\eps_2^d} \left[\abs{\tilde{x}-x}-\eps_2^{-1}\abs{y(\tilde{x})-y(x)}\right]^+
	\,d(x,\tilde{x}).
\eald
\end{align}
A more general class of such terms is defined in \eqref{sCN1}.
We show that such terms also lead to a limiting constraint equivalent to the standard Ciarlet-Ne\v{c}as condition (Theorem~\ref{thm:cnsoft}) 
while the energy minima converge as before, at least for the regularized problem with fixed $\sigma>0$  (Theorem~\ref{thm:convergence}).
This new kind of soft constraint 
makes discrete computations easier and can even be handled by standard packages (although inefficiently). 
Moreover, as we will see below, we have enough freedom to choose something going along well with particular features of the finite elements used, or 
to enforce other physically desirable properties like global injectivity even on the discrete level, and not just in the limit (Corollary~\ref{cor:invertibility}). 
Our theoretical results are complemented by numerical computations carried out for two example problems in 2d, presented in Section~\ref{sec:numerics}.

\section{The model and structural assumptions\label{sec:model}}

\subsection{The elastic energy}
As it is standard, we assume that for a deformation $y\in W^{1,p}(\Omega;\RR^d)$, $p>d$, 
the elastic energy has the form
\begin{align*}
	E^{el}(y):=\int_\Omega W(x,\nabla y)\,dx,~~\text{$y$ satisfies \eqref{ciarletnecas},}
\end{align*}
where 
\begin{align}\label{W0}
	W:\Omega\times \RR^{d\times d}\to \RR\cup \{+\infty\}~~\text{is a Carath\'eodory function,}
\end{align}
i.e., $W(x,F)$ is measurable in $x$ and continuous in $F$. Moreover, for all $F\in \RR^{d\times d}$ and all $x\in \Omega$,
\begin{align}\label{W1}
\begin{alignedat}{2}
	W(x,F)&~=~+\infty && \quad\text{if $\det F\leq 0$,}\\
	W(x,F)&~\geq~ c_1\left(\abs{F}^p+(\det F)^{-q}\right)-c_2 && \quad\text{if $\det F>0$,}
\end{alignedat}
\end{align}
with constants $q>d$ (which is necessary for \eqref{pqsd} below), $c_1>0$ and $c_2\geq 0$. 
In addition, we assume that $f$ is polyconvex, i.e., 
\begin{align}\label{W2}
\bald
  W(x,F)=h(x,m(F)),~~&\text{with a function $h$ such that}\\
	&\text{$h(x,\cdot)$ is convex for each $x$,}
\eald
\end{align}
where $m(F)\in \RR^{n(d)}$, $n(d)=\sum_{k=1}^{d} \binom{d}{k}^2$, denotes the collection of all minors of $F$, i.e., all $k\times k$ sub-determinants with $1\leq k\leq d$. 
For instance, $m(F)=(F,\det F)\in \RR^{5}$ for $d=2$ and $m(F)=(F,\cof F, \det F)\in \RR^{19}$ for $d=3$. Here, for any $d$,
$\cof F\in \RR^{d\times d}$ denotes cofactor matrix so that
$F^{-1}=(\cof F)^T (\det F)^{-1}$ whenever $F$ is invertible. 
\begin{rem}
Due to \cite{Ba77a,CiaNe87a} (for a proof of a related lower semicontinuity property of the terms in the Ciarlet-Ne\v{c}as condition also see \cite{MieRou16a}), with the assumptions \eqref{W0}--\eqref{W2}, $E^{el}$ always has a minimizer $y^*$ in $W^{1,p}(\Omega;\RR^d)$. Like all states with finite energy, it must satisfy 
$$\det \nabla y^*>0 \quad \mbox{a.e.~in~}  \Omega.$$
\end{rem}
\subsection{Approximation including penalization and higher order terms}
Our regularized approximation of $E^{el}$ is defined as follows: 
\begin{align*}
	E_{\eps,\sigma}(y):=E^{el}_{\eps_1}(y)+E^{CN}_{\eps_2}(y)+E^{reg}_{\sigma}(y),\quad\eps=(\eps_1,\eps_2).
\end{align*}
Here, the elastic energy reads
\[
	E^{el}_{\eps_1}(y):=\int_\Omega W_{\eps_1}(x,\nabla y)\,dx
\]
where $W_{\eps_1}:\Omega\times \RR^{d\times d}\to \RR$ can be any everywhere finite approximation of $W$
such that 
\begin{align}\label{Weps}
\bald
&\begin{alignedat}{2}
	&\text{$W_{\eps_1}$ is a Carath\'eodory function and polyconvex;}&&
\end{alignedat}	\\
&\begin{aligned}
	&c_3(\abs{F}^p+\max\{\eps_1,\det F\}^{-q})-c_4\\
	&\qquad \qquad \leq W_{\eps_1}(x,F) \leq W(x,F)+\eps_1 \quad  \text{for all $F\in \RR^{d\times d}$;}\\
	&W_{\eps_1}(x,F)\geq W(x,F)-\eps_1 \quad\quad  \text{if $\abs{F}\leq \frac{1}{\eps_1}$ and $\det(F)\geq \eps_1$}.
\end{aligned}
\eald
\end{align}
with constants $c_3>0$, $c_4\geq 0$.
The term $E^{CN}_{\eps_2}(y)$ represents a penalization term for the Ciarlet-Ne\v{c}as condition 
which we will discuss in detail in the next section.
As the final piece of the energy, we added the higher order term
\begin{align*}
	E^{reg}_{\sigma}(y):=\sigma\int_\Omega \abs{D^2 y}^s \,dx
\end{align*}
with some $s>1$. Primarily, we intend to study the limit $\eps\to 0$ here, 
with $\sigma>0$ fixed. The limit $\sigma\to 0$ 
would be interesting, too, but seems out of reach a the moment.

We will always work with $q,s$ admissible for the results of \cite{HeaKroe09a}, which further restricts these exponents. 
Altogether, our assumptions on the exponents can be summarized as
\begin{align}\label{pqsd}
	p>d, \quad s>d, \quad q>\frac{sd}{s-d}.
\end{align}

\begin{rem}
One possible choice for $W_{\eps_1}$ can always be obtained by replacing $h(x,\cdot)$ in 
$\eqref{W2}$ for each $x$ by the Yosida-type approximation 
\[
	h_{\eps_1}(x,A):=\min\mysetl{h(x,\tilde{A})+\frac{1}{\zeta(\eps_1)}\absb{A-\tilde{A}}^p}{
	\tilde{A} \in \RR^{n(d)}
	}
\]
with $\zeta(\eps_1)>0$ chosen small enough to obtain \eqref{Weps}; $\zeta(\eps_1)\to 0$ as $\eps_1\to 0$.
Of course, in many special cases of $h$, fully explicit approximations are possible, too.
\end{rem}
\begin{rem}
Above, we omitted force terms in the energy, although only to keep the notation short. 
Similarly, boundary conditions are missing so far. 
These issues are discussed in greater detail in Remark~\ref{rem:forces} and Remark~\ref{rem:bc} below.
\end{rem}
\begin{rem}
More general forms of
$E^{reg}_{\sigma}$ can be used as well if this is desired for modelling purposes.
The only features we actually exploit are that with $\sigma>0$ and $s$ as above, 
\begin{enumerate}
\item[(i)] $E^{reg}_{\sigma}(y)\geq \sigma\int_\Omega \abs{D^2 y}^s \,dx$,
\item[(ii)] $E^{reg}_{\sigma}(y)$ is uniformly continuous on bounded subsets of $W^{2,s}$ (cf.~Proposition~\ref{prop:SGucont}), and
\item[(iii)] $y\mapsto E^{reg}_{\sigma}(y)$, $W^{2,s}(\Omega;\RR^d)\to \RR\cup\{+\infty\}$ is sequentially lower semicontinuous with respect to weak convergence in $W^{2,s}$.
\end{enumerate}
Moreover, (i) can be further weakened to $E^{reg}_{\sigma}(y)\geq \sigma\int_\Omega \abs{D^2 y}^s \,dx-C$ with some constant $C$ which can easily be absorbed by other terms.
For (iii), it suffices to have $E^{reg}_{\sigma}(y)$ as an integral functional depending on $D^2 y$ with a convex, polyconvex or gradient polyconvex energy density. The notion of gradient polyconvexity and related results can be found in \cite{BeKruSchloe17Pa}.
\end{rem}

\section{Variants of the Ciarlet-Ne\v{c}as condition and new penalization terms\label{sec:cn}}

Our starting point for obtaining a new kind of penalization terms
is the observation that there are many equivalent ways of stating the Ciarlet-Ne\v{c}as condition \eqref{ciarletnecas}. 
For instance, if $y$ is regular enough such that the coarea formula holds, in particular for $y\in W^{1,p}$ with $p>d$ \cite{MaSwaZie03a}, 
\eqref{ciarletnecas} is equivalent to
\begin{align}\label{ciarletnecas2}
\int_{y(\Omega)} \big(N_y(z)-1\big)\,dz~=~ 0,
\end{align}
where 
\[
  N_y(z):=\#\mysetr{\tilde{x}\in \Omega}{y(\tilde{x})=z}
\]
counts the number of times $y$ (its continuous representative) reaches the point $z$ in the deformed configuration. There is self-contact at $z$ if and only if $N_y(z)>1$. 
%
Yet another equivalent way of expressing \eqref{ciarletnecas} is
\begin{align}\label{ciarletnecas3}
\int_{\Omega\cap \{x\mid N_y(y(x))>1\}} h(x) \,dx ~=~0,
\end{align}
where $h$ can be any measurable function with $h>0$ a.e.~in $\{N_y\circ y>1\}$. Notice that the choice of such a function $h$ does not matter, since \eqref{ciarletnecas3} effectively just states that $\{N_y\circ y>1\}\subset \Omega$ is a set of measure zero, and by choosing 
$h(x)=(N_y(y(x)))^{-1}(N_y(y(x))-1)$, 
\eqref{ciarletnecas3} reduces to \eqref{ciarletnecas2} by the coarea formula.

We now introduce a new class of penalization terms $E^{CN}_{\eps_2}(y)$ that -- as we will see later --
lead to a condition of the form of \eqref{ciarletnecas3} in the limit as $\eps_2\to 0$. It is defined as follows:
\begin{align}\label{sCN1}
\bald
	&E^{CN}_{\eps_2}(y):=\\
	&\qquad \frac{1}{\eps_2^\beta}
 \int_{(\Omega\times \Omega)} 
	\frac{1}{\eps_2^d} \left[g(\abs{\tilde{x}-x})-g\Big(\frac{1}{\eps_2}\abs{y(\tilde{x})-y(x)}\Big)\right]^+
	\,d(x,\tilde{x}),
\eald
\end{align}
where $[a]^+:=\max\{0,a\}$ denotes the positive part, $\beta>0$ is a constant and
\begin{align}\label{sCN1b}
\bald
	&g:[0,\infty)\to [0,\infty)~~\text{is a continuous,}\\
	&\text{strictly increasing function with $g(0)=0$.}
\eald
\end{align}
The choice of $\beta$ and $g$ is meant to give us some freedom to optimize the behavior of numerical schemes, with prototypical examples for $g$ being $g(t):=t$ or $g(t)=t^2$.

\begin{rem}\label{rem:aura}
The way $E^{CN}_{\eps_2}$ is defined, 
its integrand only contributes in 
$O(\eps_2)$-neighborhoods of the self-contact (or self-penetration) set.
More precisely, this ``aura'' 
never goes beyond a distance of
$\diam(\Omega)\operatorname{Lip}(y^{-1})\eps_2$
away from the self-contact set. 
Here, $\operatorname{Lip}(y^{-1})$ is a Lipschitz constant of the local inverse $y^{-1}$ (which is globally uniform, cf.~
 Lemma~\ref{lemIMT-biLi} below). 
When computing the integral by numerical integration, this also means that
a mesh size $h$ of this order is needed, at least near self-penetration. 
Otherwise, huge errors are likely.

In the example illustrated in Figure \ref{fig:1-4}, the radius of the aura beyond the self-contact is roughly $\eps_2$. In that particular case, 
$\operatorname{Lip}(y^{-1})$ is still quite close to $1$, and 
instead of $\diam(\Omega)$ in the estimate mentioned above, we may actually also use a number close to $1$, namely, the distance of the 
two disjoint subsets of the reference configuration (undeformed domain) where the self-overlap happens (for which $\diam(\Omega)$ is of course an upper bound).
\end{rem}


\begin{rem}
By default, all finite dimensional norms $\abs{\cdot}$ appearing in this article are assumed to be Euclidean. However,
that choice does not really matter. For instance, using a different norm inside of $g$
 in \eqref{sCN1} is possible. The proofs below are only affected insofar as all balls or annuli in
$\RR^d$ or their intersections with $\Omega$ have to be interpreted as balls (or annuli) with respect to that norm. 
Additional constants will then appear in Cauchy-Schwarz type inequalities, but that only changes the constants appearing in the results, not their general structure.
In particular, for discretizations with finite elements defined on cubes, it can be quite convenient to use $\abs{x}_\infty=\max \abs{x_i}$, $x=(x_1,\ldots,x_d)^T\in \RR^d$, instead of the Euclidean norm.
\end{rem}

\subsection{Illustration example: the penalization term for a prescribed deformation}\label{subs:example1}
All pictures of this example are displayed in Figure \ref{fig:1-4}. We assume a ``pincers'' domain $\Omega \in \RR^2$ covered in the rectangle $(-3,2) \times (-1.5, 1.5)$. Using rotated polar coordinates 
$$r=\sqrt{x_1^2+x_2^2}, \quad  t=\arctan(-x_2,-x_1),$$       
where $x=(x_1,x_2) \in \Omega$ we define the deformation in the form
$$ y(r,t)=-r (\cos(a \, t), \sin(a \, t))$$
for some parameter $a>1$. For a sufficiently high value of $a$ (here we choose $a=1.1$) both pincers parts interpenetrate and the marginal density
\begin{align*}
\bald
	d^{CN}_{\eps_2,y}(x):=
	&\frac{1}{\eps_2^\beta}
 \int_{\Omega} 
	\frac{1}{\eps_2^d} \left[g(\abs{\tilde{x}-x})-g\Big(\frac{1}{\eps_2}\abs{y(\tilde{x})-y(x)}\Big)\right]^+
	\,d\tilde{x}
\eald
\end{align*}
of $E^{CN}_{\eps_2}$
is evaluated and visualized for $\eps_2=1/2$ and $\eps_2=1/4$. In both cases we consider $\beta=1/2$ and $g(t):=t$. 
Marginal densities are 
evaluated on rectangular elements by the method of finite elements. Details on implementation are provided in Section~\ref{sec:numerics}.

\begin{figure}[htb]
    \begin{minipage}[t]{.45\textwidth}
        \centering
        \includegraphics[width=\textwidth]{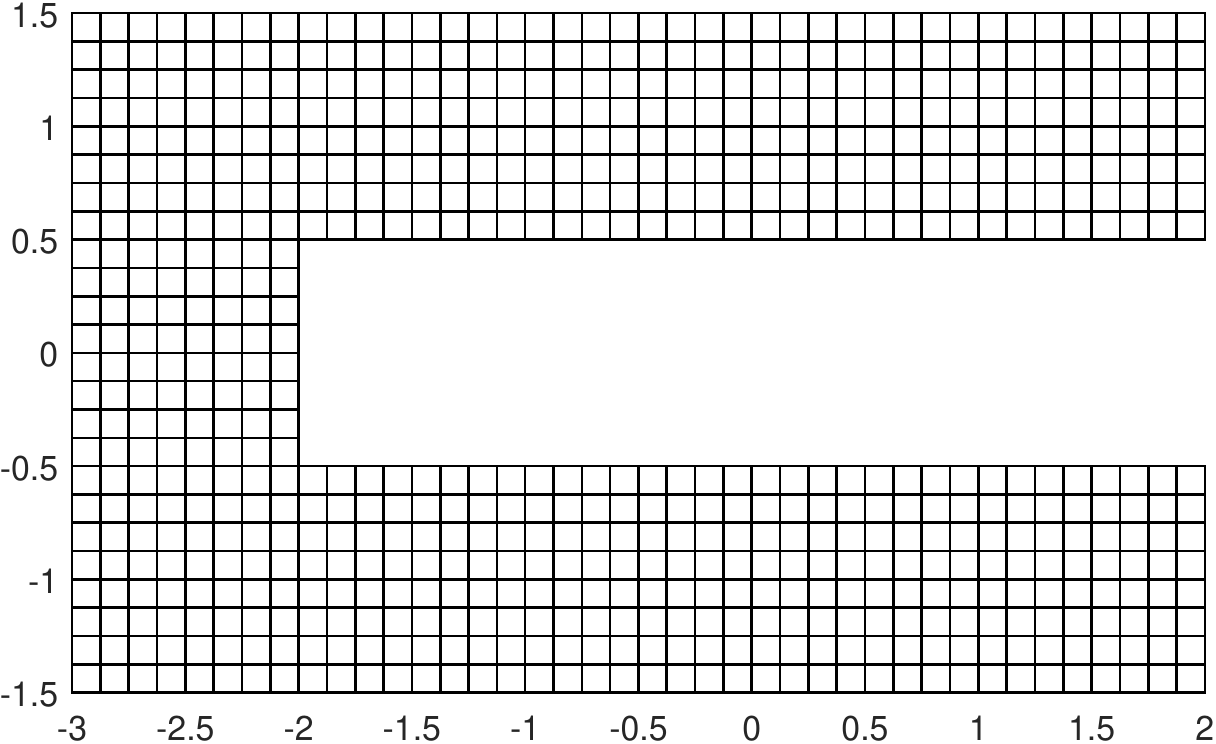}
        \subcaption{Undeformed domain.}
				
				\vspace{0.5cm}
				
				\includegraphics[width=\textwidth]{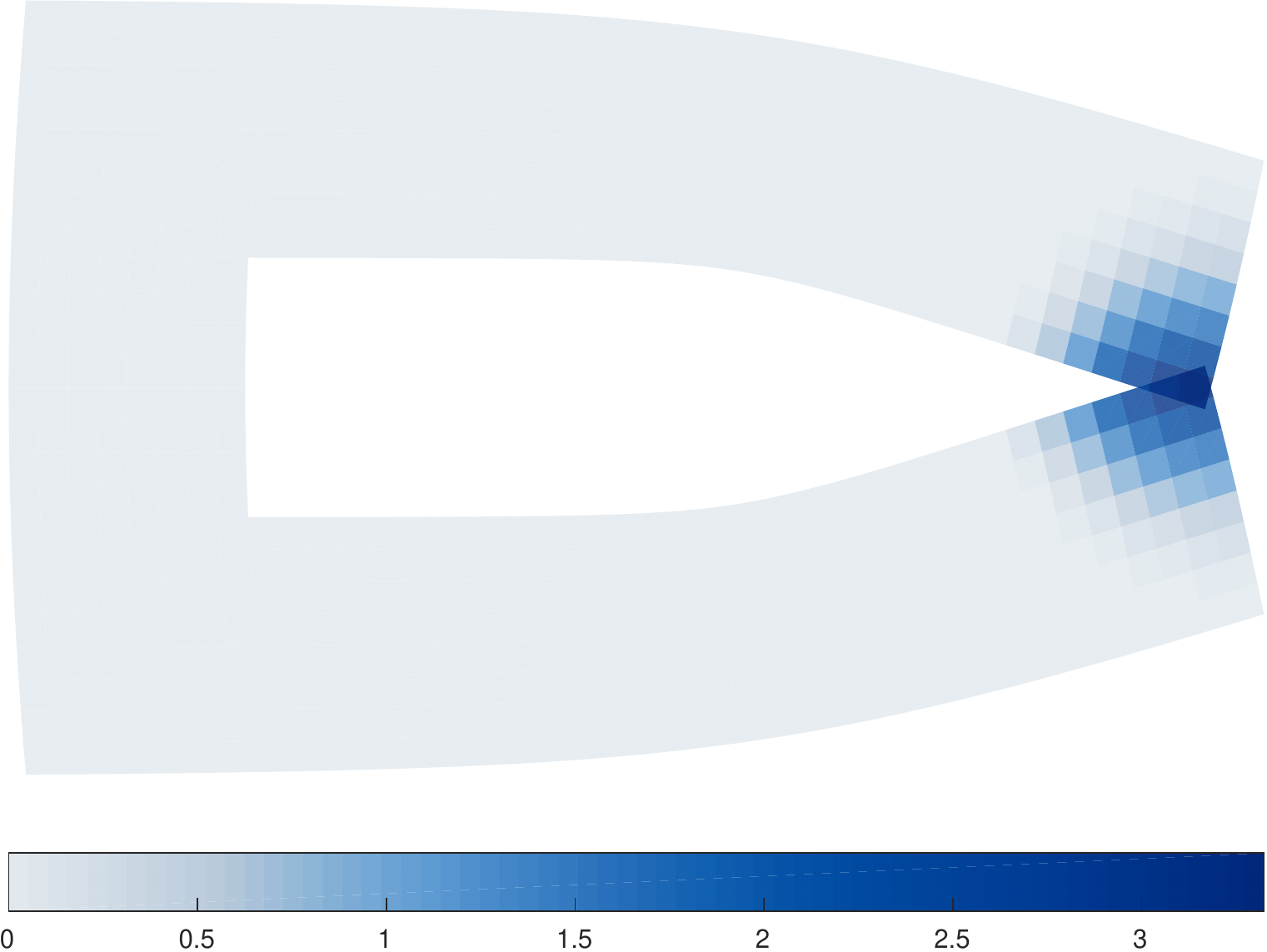}
        \subcaption{Density $d^{CN}_{\eps_2,y}(x)$ for $\epsilon_2=1/2$.}
    \end{minipage}
    \hfill
    \begin{minipage}[t]{.45\textwidth}
        \centering
        \includegraphics[width=\textwidth]{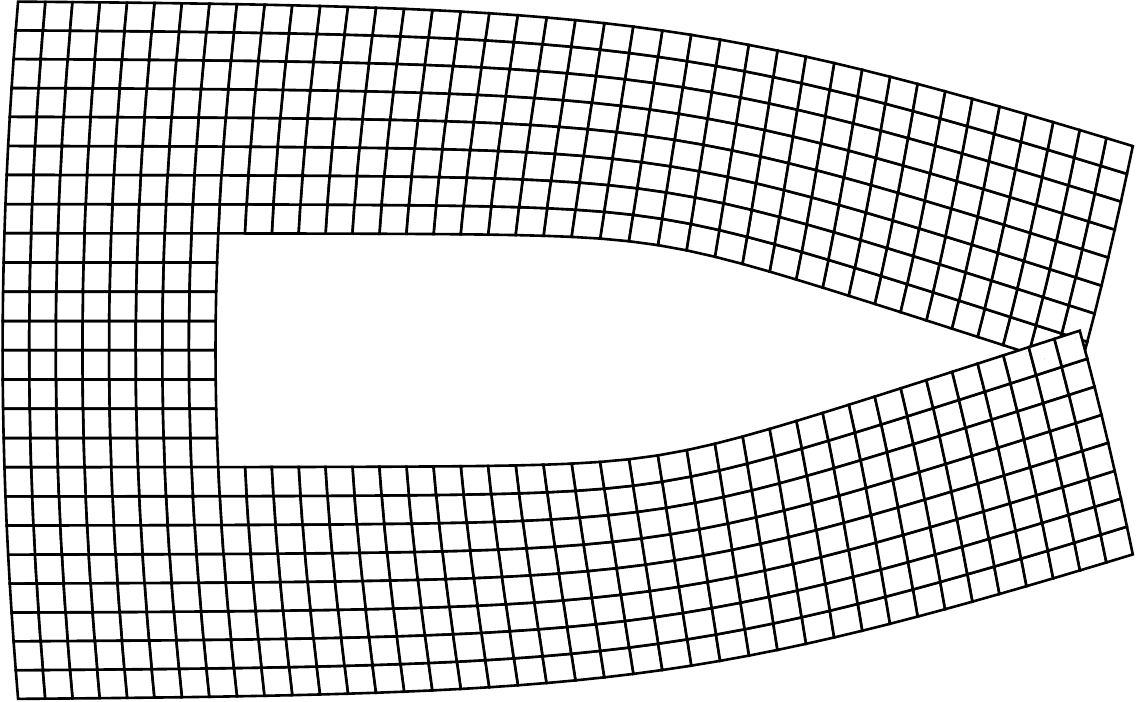}
        \subcaption{Deformed domain.}
				
				\vspace{0.5cm}
				
				\includegraphics[width=\textwidth]{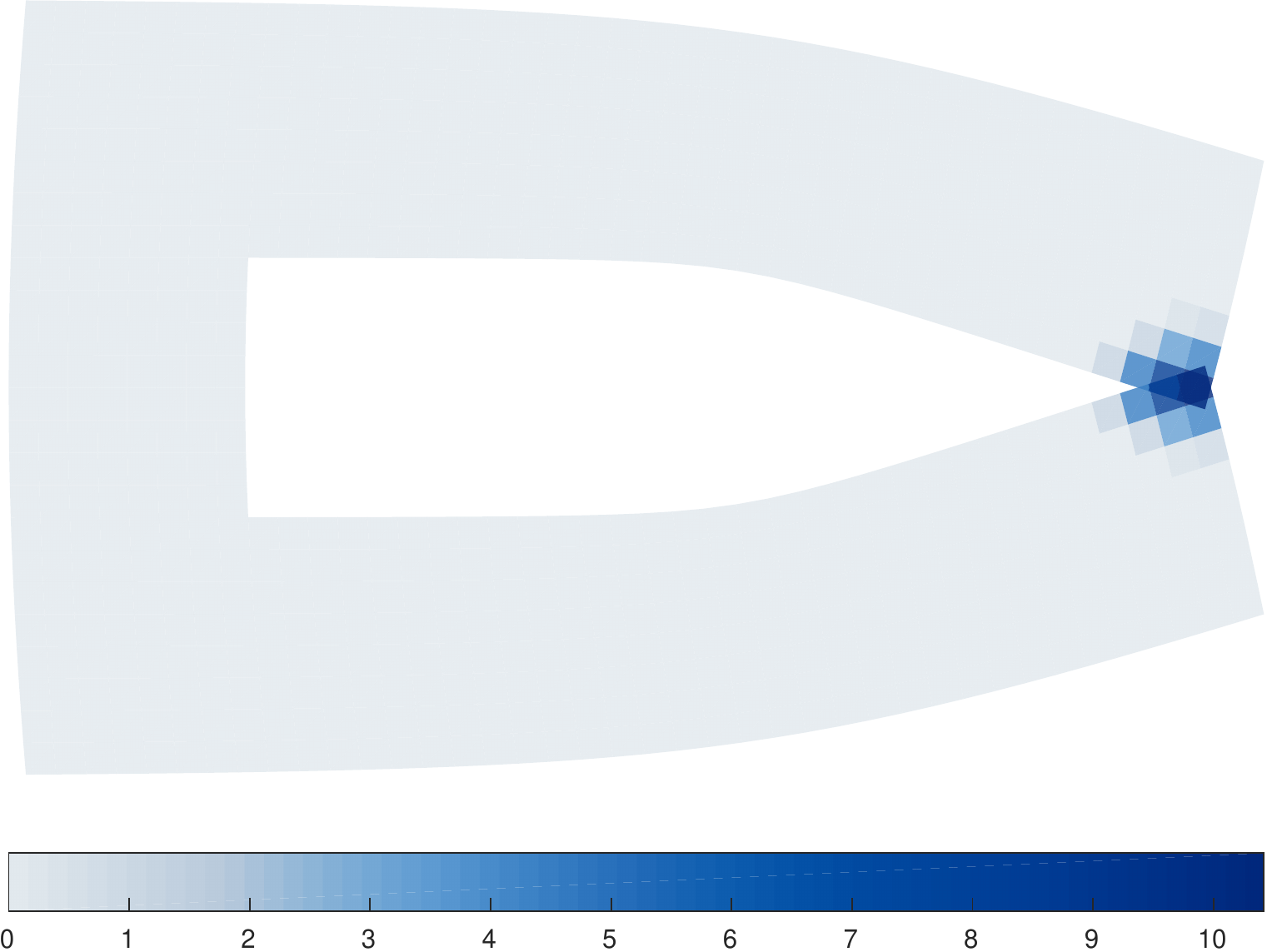}
        \subcaption{Density $d^{CN}_{\eps_2,y}(x)$ for $\epsilon_2=1/4$.}
				
    \end{minipage}  
    
    \caption{Pincers domain under given deformation.}\label{fig:1-4}
\end{figure}


\subsection{Analytic investigation of the penalty term}

We now analyze the behavior of $E^{CN}_{\eps_2}$ as $\eps_2\to 0$. 

\begin{thm}[convergence for penalty terms of type \eqref{sCN1}]\label{thm:cnsoft}
Let $\Omega\subset \RR^d$ be a bounded Lipschitz domain, 
let $E^{CN}_{\eps_2}$ be the functional defined in \eqref{sCN1} with $\beta>0$ and $g$ satisfying \eqref{sCN1b}, let $0<\alpha\leq 1$, $0<\delta,M_1,M_2$,
and let $R:=\diam(\Omega)=\sup_{x_1,x_2\in \Omega} \abs{x_1-x_2}$.
Then there exist
constants $r,a,A,\bar{\eps}>0$ only depending on $d$, $\Omega$, $\delta$, $\alpha$, $M_1$, $M_2$ and $g$ such that
for every 
$y\in C^{1,\alpha}(\Omega;\RR^d)$ with
\begin{align}\label{yconstants}
	\text{$\det \nabla y\geq \delta>0$}
	~~~\text{and}~~~
	\text{$\abs{\nabla y}\leq M_1$}~~~\text{on $\Omega$}~~~\text{and}~~~\norm{\nabla y}_{C^\alpha(\Omega)}\leq M_2
\end{align}
and every $0<\eps_2\leq \bar{\eps}$,
\begin{align}
	\eps_2^\beta E^{CN}_{\eps_2}(y)&\geq 
	a \abs{P_y(r\eps_2)}
	\label{sCN1conv1}\\
	\eps_2^\beta E^{CN}_{\eps_2}(y)&\leq 
	\bald[t]
	&	A \abs{P_y(R\eps_2)} \\
	\eald
	\label{sCN1conv2}
\end{align}
Here, 
for $s\geq 0$ (with $s=r\eps_2$ or $s=R\eps_2$ above),
\[
	P_{y}(s):=\mysetr{x\in \Omega}{
	\exists \tilde{x}\in\Omega:
	\abs{y(x)-y(\tilde{x})}\leq s~~\text{and}~~\abs{x-\tilde{x}}>\tfrac{\varrho}{2}
	},
\]
where $\varrho=\varrho(\Omega,\delta,M_1,M_2,\alpha)>0$ denotes the radius of guaranteed local injectivity from Lemma~\ref{lem:biLi} below.
\end{thm}
\begin{rem}
As a consequence of Lemma~\ref{lem:biLi}, 
$P_y(0)=\{ N_y\circ y>1\}$, i.e., it is precisely the subset in the reference configuration where global injectivity of $y$ fails.
For $s>0$, $P_{y}(s)$
is the set where $y$ almost fails to be injective up to an error of $s$.
Moreover, $P_y(0)$ is the limit $P_{y}(s)$ as $s\searrow 0$ (i.e., their intersection for all $s>0$; $P_y(s_1)\subset P_y(s_2)$ if $s_1\leq s_2$).
Thus, both the upper and the lower bound for $\eps_2^\beta E^{CN}_{\eps_2}$ in \eqref{sCN1conv1} and \eqref{sCN1conv2}, respectively, converge as $\eps_2\to 0$ by monotone convergence:
\[
	a\abs{P_y(0)}=\lim_{\eps_2\to 0} a\abs{P_y(r\eps_2)}\leq 
	\lim_{\eps_2\to 0} A\abs{P_y(R\eps_2)} = A\abs{P_y(0)}
\]
Up to the constants, these limits coincide and are functionals 
of the form of
\eqref{ciarletnecas3} with a constant integrand. 
\end{rem}
\begin{rem}
The assumption \eqref{yconstants} holds in sets with bounded energy, see Proposition~\ref{prop:yEbounded}.
\end{rem}
For the proof of the theorem, we need the following version of the Inverse Mapping Theorem with additional control, also near $\partial\Omega$. 
\begin{lem}\label{lem:biLi}
Let $\Omega\subset \RR^N$ be a bounded Lipschitz domain with local Lipschitz constants bounded by a fixed $L>0$, 
and let $y\in C^{1,\alpha}(\Omega;\RR^d)$ such that \eqref{yconstants} holds.
Then there exists an $\varrho>0$ which only depends on $\delta,M_1,M_2,\alpha$ and $\Omega$ 
such that for every $\bar{x}\in \bar\Omega$, $y$ is injective on 
$\Omega(\bar{x},\varrho):=B_{\varrho}(\bar{x})\cap \Omega$. Moreover, 
$y$ is bi-Lipschitz with explicitly known constants:
\begin{align}\label{lemIMT-biLi}
\bald
	\frac{1}{2} \frac{\delta}{M_1^{d-1}} \abs{x_1-x_2}\leq 
	\abs{y(x_1)-y(x_2)} \leq M_1 \sqrt{1+L^2} \abs{x_1-x_2} ,
\eald
\end{align}
for all $x_1,x_2\in \Omega(\bar{x},\varrho)$. 
In addition, the inverse 
$y^{-1}$ of the restriction of $y$ to $\Omega(\bar{x},\varrho)$ is of class $C^{1,\alpha}$.
\end{lem}

\begin{proof}
Since $\Omega$ is a Lipschitz domain, there exists an $R_0=R_0(\Omega)>0$
such that for all $x_0\in \partial\Omega$,
\begin{align}\label{lemIMT-0}
\bald
	&\text{in a cuboid containing $B_{R_0}(x_0)$, $\partial\Omega$ is the graph}\\
	&\text{of a Lipschitz map with constant at most $L=L(\Omega)$.}
\eald
\end{align}
For 
\[
	r_0:=\frac{1}{1+2\sqrt{1+L^2}}R_0,
\]
an explicit path connecting $x_1,x_2\in \Omega(x_0,r_0)$ in 
the larger set $\Omega(x_0,R_0)$
is given by the ``V''-shaped piecewise $C^1$-path with slope $L$ a.e.,
in $\Omega(x_0,R_0)$ below the graph representing $\partial\Omega$. The length of such a path
is at most $\sqrt{1+L^2}\abs{x_2-x_1}$.
In particular, if $\varrho\leq r_0$ and $x_1,x_2\in \Omega(x_0,\varrho)$, such a path never leaves the set $\Omega(x_0,\frac{R_0}{r_0}\varrho)$.
Hence, \eqref{lemIMT-0} implies that
\begin{align}\label{lemIMT-1}
	\bald
	&1\leq d_{x_0,\varrho}\leq \sqrt{1+L^2}\quad\text{for all $\varrho\leq r_0$},
	\eald
\end{align}
where	
\begin{align*}
	\bald	
	&d_{x_0,\varrho}:=\underset{x_1,x_2\in \Omega(x_0,\varrho)}{\sup} 
	\inf_p\mysetr{\frac{\len(p)}{\abs{x_1-x_2}}}{
	\begin{array}{l}
	p:[0,1]\to \Omega(x_0,\frac{R_0}{r_0}\varrho)\\
	\text{is piecewise $C^1$},\\ 
	\text{$p(0)=x_1$, $p(1)=x_2$}
	\end{array}}.
	\eald
\end{align*}
Notice that $d_{x_0,\varrho}$ is the worst possible ratio of intrinsic path distance and Euclidean distance in 
$\Omega(x_0,\frac{R_0}{r_0}\varrho)$ for pairs of points in the smaller set $\Omega(x_0,\varrho)$.

As a first consequence of \eqref{lemIMT-1}, $y$ is globally Lipschitz on $\Omega(x_0,\varrho)$ with a Lipschitz constant of at most 
$\norm{Dy}_{L^\infty(\Omega(x_0,\varrho))} d_{x_0,\varrho} \leq M_1\sqrt{1+L^2}$, which gives the second inequality in \eqref{lemIMT-biLi}.

For the first inequality in \eqref{lemIMT-biLi} let $\bar{x}\in \bar\Omega$.
If either $B_{\frac{1}{2}r_0}(\bar{x})\subset \Omega$ or $\bar{x}\in \partial \Omega$, we may take $x_0:=\bar x$ and 
\eqref{lemIMT-1} holds for all $\varrho\leq \frac{1}{2}r_0$.
In case we are given $\bar{x}\in \bar\Omega$ ``in between'' with 
$B_{\frac{1}{2}r_0}(\bar{x})\cap \partial\Omega\neq \emptyset$,  there always exists
$x_0=x_0(\bar{x})\in \partial\Omega$ 
such that $\Omega(\bar{x},\frac{1}{2}r_0)\subset \Omega(x_0,r_0)$
and we again have \eqref{lemIMT-1}.
In addition, for any $\bar{x}\in \bar\Omega$, all $\varrho\leq r_0$ and any pair
$x_1,x_2 \in \Omega(\bar{x},\frac{1}{2}\varrho)\subset \Omega(x_0,\varrho)$
connected with a $C^1$-path $p:[0,1]\to \Omega(x_0,\frac{R_0}{r_0}\varrho)$, $p(0)=x_1$, $p(1)=x_2$,
we have that
\[
\bald
	&\abs{y(x_1)-y(x_2)}=\abs{\int_0^1 Dy(p(t))\dot p(t)\,dt} \\
	&\geq \abs{\int_0^1 Dy(x_0)\dot p(t)\,dt} - \abs{\int_0^1 \abs{Dy(p(t))-Dy(x_0)} \abs{\dot p(t)}\,dt} \\
	&\geq \abs{\int_0^1 Dy(x_0)\dot p(t)\,dt}-  M_2 \norm{p-x_0}_{L^\infty(0,1)}^\alpha \int_0^1 \abs{\dot p(t)}\,dt \\
	&\geq \abs{Dy(x_0) (x_2-x_1)}- M_2 \left(\frac{R_0}{r_0}\varrho\right)^\alpha \len(p).
\eald
\]
Since this is true for all such paths $p$ and $d_{x_0,\varrho}\leq \sqrt{1+L^2}$ by \eqref{lemIMT-1}, we infer that
\begin{align}\label{lemIMT-2}
\bald
	&\abs{y(x_1)-y(x_2)}\\
	&\geq \abs{Dy(x_0) (x_2-x_1)}- M_2 \left(\frac{R_0}{r_0}\varrho\right)^\alpha \sqrt{1+L^2} \abs{x_1-x_2}
\eald
\end{align}
As $Dy(x_0)$ is invertible with 
$\abs{Dy(x_0)^{-1}}\leq \frac{M_1^{N-1}}{\delta}$, we also have that
\begin{align}\label{lemIMT-2a}
\bald
	\abs{Dy(x_0) (x_2-x_1)}\geq \frac{\delta}{M_1^{N-1}} \abs{x_1-x_2}. 
\eald
\end{align}
We now choose $\varrho$ small enough so that $M_2 \left(\frac{R_0}{r_0}\varrho\right)^\alpha \sqrt{1+L^2}\leq \frac{1}{2} \frac{\delta}{M_1^{N-1}}$,
and for that choice, \eqref{lemIMT-2}, \eqref{lemIMT-2a} yield that
\begin{align*}
\bald
	 \abs{y(x_1)-y(x_2)} &\geq \frac{1}{2} \frac{\delta}{M_1^{N-1}} \abs{x_1-x_2} 
\eald
\end{align*}
for all $x_1,x_2\in \Omega(\bar{x},\frac{\varrho}{2})$, proving the first inequality in \eqref{lemIMT-biLi}.
This in turn implies that $y$ is locally injective. Finally, to see the asserted $C^{1,\alpha}$-regularity of $y^{-1}$, observe that
$D y^{-1}(z)=Dy(y^{-1}(z))^{-1}$. 
Therefore, due to the Lipschitz regularity of $y^{-1}$ provided by \eqref{lemIMT-biLi},
$D y^{-1}\in C^{\alpha}$ just like $Dy$.
\end{proof}

\begin{proof}[Proof of Theorem~\ref{thm:cnsoft}] 
As before, we use the following shorthand notation for $x\in\Omega$ and $s>0$:
\[
\bald
	&\Omega(x,s):=B_{s}(x)\cap \Omega,
\eald
\]
Next, we introduce and study a few auxiliary sets related to the definition $P_y(s)$ that will be needed in the rest of the proof.

{\bf (i) Auxiliary sets related to $\mbf{P_y(s)}$: $\mbf{Q_y(s,x)}$ and $\mbf{X_y(s,x)}$.} \\
For $s\geq 0$ and $x\in P_y(s)$ let
$Q_y(s,x)$ denote the set of all admissible choices of $\tilde{x}$ in the definition of $P_y(s)$, additionally 
including those that are close to $x$:
\[
	Q_y(s,x):=\mysetr{\tilde{x}\in \Omega}{\abs{y(x)-y(\tilde{x})}\leq s}.
\]
We claim that for $s$ small enough, $Q_y(s,x)$ is separated into subsets of 
small balls that are pairwise far apart: 
For every
\[
	s<K\varrho \quad\text{where}~~K:=\frac{\delta}{4 M_1^{d-1}},
\] 
we have that
\begin{align}\label{thmcns1-20}
\bald
	Q_y(s,x)\subset (\Omega\setminus \Omega(x_0,\varrho))  \,\cup\, \Omega(x_0,K s)
	\quad\text{for all $x_0\in Q_y(s,x)$}.
\eald
\end{align}
For the proof of \eqref{thmcns1-20}, take any
$z\in \Omega$, $z\notin (\Omega\setminus \Omega(x_0,\varrho))  \cup \Omega(x_0,K s)$. Then
\[
	\frac{4 M_1^{d-1}}{\delta} s< \abs{z-x_0}\leq \varrho,
\]
and the uniform bi-Lipschitz property \eqref{lemIMT-biLi} 
entails that
\[
	2s< \abs{y(z)-y(x_0)}.
\]
Since $x_0\in Q_y(s,x)$ and thus $\abs{y(x)-y(x_0)}\leq s$, we infer that
$s< \abs{y(z)-y(x)}$, i.e., $z\notin Q_y(s,x)$.

As a consequence of \eqref{thmcns1-20},
for $s$ small enough as above, we can select
a finite set $X_y(s,x)$ 
such that 
\begin{align}\label{thmcns1-21}
\bald
	&X_y(s,x) \subset Q_y(s,x),\quad x\in X_y(s,x),\\
	&\abs{x_1-x_2}\geq \varrho\quad\text{for all $x_1,x_2\in X_y(s,x)$, $x_1\neq x_2$},
\eald
\end{align}
and $Q_y(s,x)$ is contained is a disjoint union of small balls centered at points in $X_y(s,x)$:
\begin{align}\label{thmcns1-22}
\bald
	Q_y(s,x)\subset \bigcup_{x_0\in X_y(s,x)} \Omega\left(x_0, K s \right).
\eald
\end{align}
Finally, notice that by the definition of $P_y(s)$, 
$Q_y(s,x)\setminus \Omega(x,\varrho)\neq \emptyset$. Therefore, $X(s,x)$ always contains at least one more point besides $x$:
\begin{align}\label{thmcns1-23}
\bald
	\sharp X(s,x)\geq 2\quad\text{for all $x\in P_y(s)$}.
\eald
\end{align}


\noindent {\bf (ii) Splitting $\mbf{E^{CN}_{\eps_2}}$.}\\
For any $s\geq 0$, splitting 
$\Omega=P_y(s) \cup (\Omega\setminus P_y(s))$
gives that
\begin{align}\label{thmcns1-10}
\bald
&\eps_2^\beta E^{CN}_{\eps_2}(y) = \int_{P_y(s)} \frac{1}{\eps_2^d} J_{y,\eps_2}(x)\,dx+
\int_{\Omega\setminus P_y(s)} \frac{1}{\eps_2^d} J_{y,\eps_2}(x)\,dx,
\eald
\end{align}
where 
\begin{align*}
&\bald
J_{y,\eps_2}(x)&:= \int_{\Omega} 
\left[g(\abs{\tilde{x}-x})-g\Big(\frac{\abs{y(\tilde{x})-y(x)}}{\eps_2}\Big)\right]^+
	\,d\tilde{x}.\\
\eald
\end{align*}
Below, we estimate the two terms on the right hand side of \eqref{thmcns1-10} separately, for suitable choices of $s$ depending on $\eps_2$.

{\bf (iii) Proof of \eqref{sCN1conv1}.} \\
We use $s:=r\eps_2$ in \eqref{thmcns1-10}, with
some constant $0<r\leq 1$ to be determined later.
Since $J_{y,\eps_2}\geq 0$, we get that
\begin{align}\label{thmcns1-1a1}
\bald
\eps_2^\beta E^{CN}_{\eps_2}(y) & \geq \int_{P_y(r\eps_2)} \frac{1}{\eps_2^d} J_{y,\eps_2}(x)\,dx.
\eald
\end{align}

The integrand of $J_{y,\eps_2}$ is also non-negative, and therefore, 
using \eqref{thmcns1-22},
for each $x\in P_y(r\eps_2)$ 
we also have that
\begin{align}\label{thmcns1-1a2}
\bald
	J_{y,\eps_2}(x)
 \geq 
\sum_{x_0 \in X(r\eps_2,x)\setminus \{x\}} 
I_{y,x,x_0}(\eps_2)
\eald
\end{align}
where 
\begin{align}\label{thmcns1-1b}
  I_{y,x,x_0}(\eps_2):=
	\int_{\Omega(x_0,r\eps_2) }
\left[g(\abs{\tilde{x}-x})-g\Big(\frac{1}{\eps_2}\abs{y(\tilde{x})-y(x)}\Big)\right]^+
	\,d\tilde{x}.
\end{align}
We will now proceed to estimate $I_{y,x,x_0}(\eps_2)$ for all $x\in P_y(r\eps_2)$ and $x_0\in X(r\eps_2,x)\setminus \{x\}$,
which by \eqref{thmcns1-21} in particular implies that $\abs{x-x_0}\geq \varrho$.
For $\tilde{x}\in \Omega(x_0,r\eps_2)$,
the latter yields that
\[
	\abs{\tilde{x}-x}\geq \abs{x-x_0}-\abs{x_0-\tilde{x}} \geq \frac{\varrho}{2}
\]
as long as $2 r\eps_2\leq  \varrho$. 
Later, it will be convenient to also have that
$r\eps_2\leq \frac{1}{2}R=\frac{1}{2}\diam(\Omega)$.
As long as $r\leq 1$,
it altogether suffices if
\[
	\eps_2\leq \bar\eps:=\min\left\{\frac{\varrho}{2},\frac{R}{2}\right\}
\]
Moreover, recall that $x_0\in X(r\eps_2,x)\subset Q(r\eps_2,s)$. By the definition of $Q(r\eps_2,s)$ in step (i),
this entails that $\abs{y(x_0)-y(x)}\leq r\eps_2$, and consequently,
\[
\bald
	\abs{y(\tilde{x})-y(x)} &\leq \abs{y(x_0)-y(x)}+\abs{y(\tilde{x})-y(x_0)} \\
	& \leq r\eps_2+\abs{y(\tilde{x})-y(x_0)}
	\leq r\eps_2+M_1\sqrt{L^2+1} \abs{\tilde{x}-x_0},
\eald
\]
where we also used the local Lipschitz continuity of $y$ given by \eqref{lemIMT-biLi}.
With these observation and the monotonicty of $g$, both expressions in $g$ can be estimated 
and in this way, \eqref{thmcns1-1b} implies that
\begin{align}\label{thmcns1-2}
\bald
I_{y,x,x_0}(\eps_2)
&\geq  \int_{\Omega(x_0,r\eps_2) }
\left[g\Big(\frac{\varrho}{2}\Big)-g\Big(r+\frac{1}{\eps_2}M_1\sqrt{L^2+1} \abs{\tilde{x}-x_0}\Big)\right]^+
	\,d\tilde{x} \\
&\geq \int_{\Omega(x_0,r\eps_2) } 
\left[g\Big(\frac{\varrho}{2}\Big)-g\Big(r+r M_1\sqrt{L^2+1}\Big)\right]^+
	\,d\tilde{x}.
\eald
\end{align}	
Here, $r+r M_1\sqrt{L^2+1}\leq \tfrac{\varrho}{4}$ for
\[
  r:=\min\left\{\tfrac{\varrho}{4}(1+M_1\sqrt{L^2+1})^{-1},1\right\}.
\]
Substituting $t:=r\eps_2\leq \bar\eps\leq \frac{1}{2}R$, we conclude that
\begin{align}\label{thmcns1-3}
\bald
I_{y,x,x_0}(\eps_2)
&\geq a \eps_2^d
\eald
\end{align}	
with the constant
\[
	0<a:=r^d\left(g\Big(\frac{\varrho}{2}\Big)-g\Big(\frac{\varrho}{4}\Big)\right)
	\inf_{t,x_0} \mysetl{\frac{\abs{\Omega(x_0,t)}}{t^d}}{\begin{array}{l}
	0<t\leq \frac{1}{2}R,\\x_0\in \Omega
	\end{array}}.
\]
Notice that the infimum above is a geometric constant which only depends on $\Omega$. It is determined by
the smallest possible the volume fractions $\abs{\Omega\cap B_t(x_0)}/\abs{B_t(x_0)}$. 
Such fractions are bounded away from zero because $\Omega$, being a Lipschitz domain, satisfies an interior cone condition.

Combined with \eqref{thmcns1-1a1}, \eqref{thmcns1-1a2} and \eqref{thmcns1-23}, \eqref{thmcns1-3} yields \eqref{sCN1conv1}.

\noindent {\bf (iv) Proof of \eqref{sCN1conv2}.} \\
This time, we use \eqref{thmcns1-10} with $s=R\eps_2$, $R=\diam{\Omega}$,
and distinguish the cases $x\in P_y(R\eps_2)$ and $x\in \Omega\setminus P_y(R\eps_2)$.

\noindent {\bf Case 1: $\mbf{x\in \Omega\setminus P_y(R\eps_2)}$.} 
We claim that for such $x$, $J_{y,\eps_2}(x)=0$ for sufficiently small $\eps_2$. If $x\in \Omega\setminus P_y(R\eps_2)$ then 
for all $\tilde{x}\in\Omega$,
\[
  \abs{y(\tilde{x})-y(x)}> R\eps_2\quad\text{or}\quad\abs{\tilde{x}-x}<\frac{\varrho}{2}.
\]
In the former case, the integrand in $J_{y,\eps_2}(x)$ vanishes
since $\abs{x-\tilde{x}}\leq \operatorname{diam}{\Omega}=R$.
In the latter case, Lemma~\ref{lem:biLi} can be applied, and
due to the monotonicity of $g$ and the lower bound in \eqref{lemIMT-biLi},
the integrand in $J_{y,\eps_2}(x)$ vanishes again, at least
if
\begin{align*}
\bald
  \eps_2\leq \tilde{\eps}, \quad \tilde{\eps}:=\frac{\delta}{2 M_1^{d-1}}.
\eald
\end{align*}
Hence, 
\begin{align}\label{thmcns1-12}
\bald
  J_{y,\eps_2}(x)=0\quad \text{if $x\in \Omega\setminus P_y(R\eps_2)$ and $\eps_2\leq \tilde{\eps}$}.
\eald
\end{align}

\noindent {\bf Case 2: $\mbf{x\in P_y(R\eps_2)}$.} 
Let
\[
	\eps_2\leq \bar\eps:= \frac{\varrho}{R} K,~~\text{with}~~K=\frac{\delta}{4 M_1^{d-1}}~~\text{as in \eqref{thmcns1-20}}
\]
($\bar{\eps}$ here differs from its old namesake).
Since $\abs{y(\tilde{x})-y(x)}<R\eps_2$ if and only if $\tilde{x}\in Q_y(R\eps_2,x)$, the integrand 
in $J_{y,\eps_2}(x)$ vanishes for all other $\tilde{x}$:
\[
  \left[g(\abs{\tilde{x}-x})-g\Big(\frac{\abs{y(\tilde{x})-y(x)}}{\eps_2}\Big)\right]^+
	\leq 
	\left[g(R)-g\Big(\frac{\abs{y(\tilde{x})-y(x)}}{\eps_2}\Big)\right]^+
	= 0
\]
if $\abs{y(\tilde{x})-y(x)}\geq R\eps_2$, since $g$ is increasing.
For $\eps_2\leq \bar{\eps}$ and $x\in P_y(R\eps_2)$, 
we can therefore use \eqref{thmcns1-22}
to estimate $J_{y,\eps_2}(x)$
as follows:
\begin{align}\label{thmcns1-24}
\bald
&J_{y,\eps_2}(x)\\
&\leq \int_{Q_y(R\eps_2,x)}
\left[g(R)-g\Big(\frac{\abs{y(\tilde{x})-y(x)}}{\eps_2}
\Big)\right]^+ \,d\tilde{x},\\
&\leq
\sum_{x_0\in X_y(x,R\eps_2)} 
\int_{\Omega(x_0, K R \eps_2)} 
\left[g(R)-g\Big(\frac{\abs{y(\tilde{x})-y(x)}}{\eps_2}
\Big)\right]^+ \,d\tilde{x},\\
&\leq  
\sum_{x_0\in X_y(x,R\eps_2)}
\abs{\Omega(x_0, K R \eps_2)} g(R).\\
&\leq (\sharp X_y(x,R\eps_2)) (KR)^d \abs{B_{1}(0)} g(R).\\
\eald
\end{align}
This is bounded by a suitable constant $A$ because
$\sharp X_y(x,R\eps_2)$, the number of elements of $X_y(x,R\eps_2)$, is bounded by a constant 
only depending on $\varrho$ and $R=\operatorname{diam}{\Omega}$,
as a consequence of \eqref{thmcns1-21}.
\end{proof}
Theorem~\ref{thm:cnsoft} provides additional insights on the behavior of $E^{CN}_{\eps_2}$:
\begin{cor}\label{cor:whenECNvanishes}
In the situation of Theorem~\ref{thm:cnsoft}, let $\eps_2\leq \bar{\eps}$ and suppose in addition that $y$ is more than a distance of $R\eps_2$ away from any self-contact, i.e.,
\begin{align}\label{corwECNv-1}
	\abs{y(x_1)-y(x_2)}> R\eps_2\quad\text{for all}~\abs{x_1-x_2}> \tfrac{\varrho}{2},
\end{align}
with $R=\operatorname{diam}\Omega$ as before.
Then $E^{CN}_{\eps_2}(y)=0$. 
\end{cor}
\begin{proof}
This is a direct consequence of \eqref{sCN1conv2} and the definition of $P_y(R\eps_2)$: \eqref{corwECNv-1} implies that $P_y(R\eps_2)=\emptyset$.
\end{proof}
Another interesting consequence of Theorem~\ref{thm:cnsoft} is
\begin{cor}[Global invertibility for finite $\eps_2$]\label{cor:invertibility}
Suppose that the assumptions of Theorem~\ref{thm:cnsoft} hold and let $C>0$.
If $\beta>d$ in \eqref{sCN1}
then there exists a constant $0<\tilde{\eps}\leq \bar\eps$ which only depends on $\beta$, $C$, $d$, $\Omega$, $\delta$, $\alpha$, $M_1$, $M_2$ and $g$,
such that for all $\eps_2<\tilde{\eps}$ and all $y\in C^{1,\alpha}(\Omega;\RR^d)$ satisfying \eqref{yconstants}, 
\begin{align}\label{yECN-bound}
	E^{CN}_{\eps_2}(y)\leq C\quad\text{implies that $y$ is globally injective.}
\end{align}
\end{cor}
\begin{proof}
We will prove \eqref{yECN-bound} indirectly.
Suppose that $y$ is not globally injective, i.e., $y(x_1)=y(x_2)$ for a pair of points $x_1,x_2\in\Omega$, $x_1\neq x_2$.
In view of \eqref{sCN1conv1}, 
it suffices to show that then
\begin{align}\label{corGI-1}
	\eps_2^{-\beta} a \abs{P_y(r\eps_2)}>C\quad\text{for all $\eps_2<\tilde{\eps}$}
\end{align}
with a suitable choice of $\tilde{\eps}>0$.
We claim that
\begin{align}\label{corGI-2}
	\abs{P_y(r\eps_2)}\geq c \eps_2^d\quad\text{for all $\eps_2\leq \hat{\eps}$}
\end{align}
with constants $\hat{\eps}>0$, $c>0$ yet to be determined.
From \eqref{corGI-2}, we immediately
get \eqref{corGI-1} with $\tilde{\eps}:=\min\big\{\hat{\eps},\big(\frac{ca}{C}\big)^{\frac{1}{\beta-d}}\big\}>0$. 

To prove \eqref{corGI-2},
first notice that as a Lipschitz domain, $\Omega$ satisfies an interior cone condition, i.e.,
there is a (cut off) cone of the form 
\[
	V=B_\mu(0)\cap \{z\in\RR^d\mid z\cdot e>\nu \abs{z}\}
\]
(with a fixed unit vector $e\in \RR^d$ and constants $\nu<1$, $\mu>0$)
which only depends on $\Omega$ such that
for each $x\in\Omega$, $x+Q V\subset \Omega$ with a suitable rotation $Q=Q(x)\in SO(d)$.
In particular, there is $Q_1,Q_2\in SO(d)$ such that
$x_1+Q_1 V\subset \Omega$ and $x_2+Q_2 V\subset \Omega$.
By Lemma~\ref{lem:biLi}, we know that $\abs{x_1-x_2}\geq \varrho$, and therefore $x_1,x_2\in P_y(0)$.
By the local Lipschitz continuity \eqref{lemIMT-biLi} of $y$ with constant $M_1 \sqrt{1+L^2}$ 
and the definition of $P_y(r\eps_2)$ in Theorem~\ref{thm:cnsoft}, 
we see that as a consequence, for $j=1,2$,
\begin{align}\label{corGI-3}
  (x_j+Q_j V)\cap B_{\lambda \eps_2}(x_j)~\subset~ P_y(r\eps_2),
	\quad\text{with }\lambda:=\frac{r}{M_1 \sqrt{1+L^2}},
\end{align}
provided that $\eps_2\leq \bar\eps$ and $\lambda \eps_2\leq \varrho$. 
If $\lambda \eps_2\leq \frac{\varrho}{2}$, we also know that $B_{\lambda \eps_2}(x_1)$ and $B_{\lambda \eps_2}(x_2)$ are disjoint.
Since $\abs{(x_j+Q_j V)\cap B_{\lambda \eps_2}(x_j)}=\frac{\abs{V}}{\abs{B_\mu(0)}}(\lambda r)^d$ as long as $\lambda \eps_2\leq \mu$,
\eqref{corGI-3} entails \eqref{corGI-2}
with $c:=2\frac{\abs{V}}{\abs{B_\mu(0)}}\lambda^d$, for all 
$\eps_2<\hat{\eps}:=\min\big\{\bar{\eps},\frac{\varrho}{2\lambda},\frac{\mu}{\lambda}\big\}$.
\end{proof}
\begin{rem}
The proof of Corollary~\ref{cor:invertibility} also works if $x_1,x_2\in \partial\Omega$, and we take $y(x_1)$ and $y(x_2)$ as the the 
uniquely determined values of the continuous extension of $y\in C^{1,\alpha}$ to $\bar\Omega$. Hence, self-contact on the surface is also prevented for all $\eps_2$ small enough. In fact, one can see with similar arguments that whenever 
$\beta>d$, a universal bound on the penalty term $E^{CN}_{\eps_2}(y)$ as in \eqref{yECN-bound} even enforces
a positive minimal distance between different pieces of the body's surface (different in the sense that they are not closer than the radius $\varrho$ of local invertibility in the reference configuration). This minimal distance converges to zero as $\eps_2\to 0$.
\end{rem}

In the final piece of this section, we discuss the stability of $E^{CN}_{\eps_2}(y)$ with respect to perturbations in $y$. 
For fixed $\eps_2$, $E^{CN}_{\eps_2}$ is obviously continuous in $L^\infty$, but 
that continuity is not uniform in the limit $\eps_2\to 0$.
From Theorem~\ref{thm:cnsoft} and the definition of the sets $P_y(r\eps_2)$, $P_y(R\eps_2)$ 
we can infer that 
$E^{CN}_{\eps_2}(y)$ does not change too much if $y$ is replaced by some perturbed deformation $z$ 
with $\norm{y-z}_{L^\infty}\leq \frac{r}{3}\eps_2$, because then 
$P_{y}(\frac{r}{3}\eps_2)\subset P_{z}(r\eps_2) \subset P_{y}(\frac{5}{3}r\eps_2)$ (and 
similar inclusions also hold with $R$ instead of $r$). Here, $z$ of course may depend on $\eps_2$. 
However, it is important to be able to handle also perturbations that are small but not controlled by $\eps_2$: 
\begin{prop}\label{prop:CNstab}
In the situation of Theorem~\ref{thm:cnsoft}, 
suppose that $y,z\in C^{1,\alpha}(\Omega;\RR^d)$ both satisfy \eqref{yconstants}. 
Then for every $0<\gamma<\frac{\varrho}{2}$ (with $\varrho$ from Lemma~\ref{lem:biLi}), there exists a constant $\lambda>0$ 
such that
\begin{align}\label{pCNs-1}
	\tilde{P}^{(\gamma)}_y(0)
	\subset
	P_z(0)\quad\text{if $\norm{y-z}_{L^\infty}\leq \lambda$,}
\end{align}
where 
\[
	\tilde{P}^{(\gamma)}_y(0):=\mysetr{x_1\in\Omega}{
	\begin{aligned}[c]
	&\text{$\exists x_2\in \Omega$ with $\dist{x_2}{\partial\Omega}>\gamma$},\\
	&\text{$y(x_1)=y(x_2)$ and $\abs{x_1-x_2}>\tfrac{\varrho}{2}+\gamma$}
	\end{aligned}
	}.
\]
Here, $\lambda$ may depend on $\gamma$ and the constants appearing in Theorem~\ref{thm:cnsoft} but not on $y$, $z$ or $\eps_2$.
\end{prop}
\begin{rem}\label{rem:CNstab}
Since $P_z(0)\subset P_z(r\eps_2)$, \eqref{pCNs-1} and \eqref{sCN1conv1} in particular imply that
\begin{align}\label{pCNs-2}
	E^{CN}_{\eps_2}(z) 
	\geq a \eps_2^{-\beta} \absn{\tilde{P}^{(\gamma)}_y(0)}\quad\text{if $\norm{y-z}_{L^\infty}\leq \lambda$}.
\end{align}
Moreover, $\tilde{P}^{(\gamma)}_y(0)$ is always an open set (because if $y(x_1)=y(x_2)$, then $y$ also self-intersects
on whole neighborhoods of $x_1$ and $x_2$ since $y$ is locally bi-Lipschitz due to Lemma~\ref{lem:biLi}). 
Therefore, whenever $P_y(0)\neq \emptyset$ 
we can find $\gamma=\gamma(y)>0$ such that $P^{(\gamma)}_y(0)\neq \emptyset$ and thus
$\absn{P^{(\gamma)}_y(0)}>0$, and the right hand side of the inequality in \eqref{pCNs-2} then blows up as $\eps_2\to 0$. 
Hence, only deformations $y$ with $P_y(0)=\emptyset$ (i.e., $y$ is globally invertible) can be reached in the limit along a sequence for which $E^{CN}_{\eps_2}$ remains bounded.
\end{rem}
\begin{proof}[Proof of Proposition~\ref{prop:CNstab}]
Let $x_1\in \Omega$ with $x_1\in \tilde{P}^{(\gamma)}_y(0)$.
By definition of $\tilde{P}^{(\gamma)}_y(0)$, there exists $x_2\in \Omega$ with $\dist{x_2}{\partial\Omega}>\gamma$, 
$y(x_1)=y(x_2)$ and $\abs{x_1-x_2}>\frac{\varrho}{2}+\gamma$.
Both $y$ and $z$ are locally bi-Lipschitz due to Lemma~\ref{lem:biLi}, 
and the constants explicitly given in \eqref{lemIMT-biLi} do not depend on $y$ or $z$.
Hence, a whole neighborhood of $y(x_1)=y(x_2)$ is contained in $y(\Omega)$.
More precisely, 
\[
	B_{\tau}(y(x_2))\subset y(B_\gamma(x_2)) \subset y(\Omega)\quad
	\text{where $\tau:=\frac{\gamma}{L_{y^{-1}}}$}.
\]
Here, $L_{y^{-1}}\geq 1$ can be any Lipschitz constant of the local inverse $y^{-1}$ of $y$ near $x_2$, for instance
$L_{y^{-1}}:=\max\big\{1,\tfrac{2 M_1^{d-1}}{\delta}\big\}$ is admissible by \eqref{lemIMT-biLi}, and this particular choice is also independent of $x_2$ and $y$.
Analogously, 
\[
	B_{\tau}(z(x_2))\subset z(B_\gamma(x_2)).
\]
Therefore,
for every $z$ with $\abs{y(x_i)-z(x_i)}<\gamma:=\frac{1}{2}\tau$, $i=1,2$, we obtain that
\[
	z(x_1)\in 
	B_\gamma(y(x_2))
	\subset 
	B_{\tau}(z(x_2))
	\subset z(B_\gamma(x_2)).
\]
This implies that $x_1\in P_z(0)$: There exists $\tilde{x}_2\in B_\gamma(x_2) \subset \Omega$ such that
$z(x_1)=z(\tilde{x}_2)$ and $\absn{x_1-\tilde{x}_2}\geq \abs{x_1-x_2}-\gamma > \frac{\varrho}{2}$.
\end{proof}

\section{Convergence of energies\label{sec:main}}

In this section, for $y\in W^{1,p}(\Omega;\RR^d)$,
we prove that in the limit as $\eps=(\eps_1,\eps_2)\to 0$, the penalized energy
\begin{align*}
	E_{\eps,\sigma}(y)=\left\{\begin{alignedat}[c]{2}
	&E^{el}_{\eps_1}(y)	+E^{reg}_{\sigma}(y)+E^{CN}_{\eps_2}(y)\quad && \text{if $y\in W^{2,s}$},\\
		&+\infty && \text{else,}
\end{alignedat}\right.	
\end{align*} 
with
\[
	E^{el}_{\eps_1}(y):= \int_\Omega W_{\eps_1}(x,\nabla y)\,dx,
\]
converges to the original energy
\begin{align*}
	E_{\sigma}(y)=\left\{\begin{alignedat}[c]{2}
	  &E^{el}(y) +E^{reg}_{\sigma}(y)\quad && \text{if $y\in W^{2,s}$ and \eqref{ciarletnecas} holds},\\
		&+\infty && \text{else,}
	\end{alignedat}\right.
\end{align*} 
which includes the Ciarlet-Ne\v{c}as condition \eqref{ciarletnecas} as a built-in constraint.
Here, recall that
\[
	E^{el}_{\eps_1}(y)= \int_\Omega W_{\eps_1}(x,\nabla y)\,dx,\quad
	E^{el}(y)=\int_\Omega W(x,\nabla y)\,dx.
\]
In addition, we also consider the convergence of discrete Galerkin approximations. 
For that, let
$h>0$ (typically a mesh size) and let $Y_h$ denote associated finite dimensional subspaces of 
$(W^{2,s}\cap W^{1,p})(\Omega;\RR^d)$ (typically $Y_h\subset W^{2,\infty}$) such that
the maximal approximation error $\cE(h)$ satisfies
\begin{align}\label{fespaceapprox}
	\cE(h):=\sup_{y\in W^{2,s}} \inf_{y_h\in Y_h}\big(\{\norm{y-y_h}_{W^{2,s}\cap W^{1,p}}\big)
	\underset{h\to 0}{\To} 0.
\end{align}

The corresponding finite dimensional approximations of $E_{\eps,\sigma}$ are
\begin{align*}
	E^h_{\eps,\sigma}(y):=\left\{\begin{alignedat}[c]{2}
	&E^{el}_{\eps_1}(y)+E^{reg}_{\sigma}(y)+E^{CN}_{\eps_2}(y)\quad && \text{if $y\in Y_h$},\\
		&+\infty && \text{else,}
\end{alignedat}\right.	
\end{align*} 
As defined, $E^h_{\eps,\sigma}$ is assumed to be exact on $Y_h$. In this context, we will not discuss the question of how to of calculate the integrals in $E^h_{\eps,\sigma}$ in practice. The easiest possible approach is of course based on additional approximations using standard methods in numerical integration. For our analysis, additional errors terms that might appear at this stage
do not matter as long as they still converge to zero as $(h,\eps)\to 0$. 
However, it is useful to optimize the evaluation of the double integral in $E^{CN}_{\eps_2}$ for performance reasons, since only small neighborhoods of the self-contact set (or any almost self-contact) actually contribute.
\begin{rem}
Artificially assigning the value $+\infty$ in the definitions of the functionals 
is just a way of encoding a restricted class of admissible functions. 
Be warned that there are still other ``inadmissible'' deformations with infinite energy
in case of $E_{\sigma}$, namely any $y\in W^{2,s}\cap W^{1,p}$ for which
$\int_\Omega W(x,\nabla y)\,dx=+\infty$ because $\det\nabla y$ 
is too close to zero or even non-positive on a non-negligible set. 
\end{rem}
\begin{rem}[additional force terms]  \label{rem:forces} 
As already briefly mentioned, we did not add any terms corresponding to exterior forces, but only to keep the notation short. Since 
we actually prove $\Gamma$-convergence, our results are stable with respect to the addition of any term that is continuous with respect to 
the topology used for the states in the $\Gamma$-limit (see \cite{Dal93B}, e.g.). For us, that is the weak topology of $W^{2,s}$
(or the weak topology of $W^{1,p}$, which is a weaker topology but still leads to the same result for fixed $\sigma>0$). 
Continuous pertubations in the weak topology of $W^{1,p}$ in particular include linear body force terms like
\begin{align}\label{forces-body}
	\int_\Omega y\cdot g_{\rm body}\,dx,\quad \text{with a }g_{\rm body}\in L^1(\Omega;\RR^d).
\end{align} 
Similarly, one can add linear boundary force terms like
\begin{align}\label{forces-surface}
	\int_{\partial\Omega} y\cdot g_{\rm surface}\,d\cH^{d-1}(x),\quad \text{with a }g_{surface}\in L^1(\partial\Omega;\RR^d),
\end{align} 
where the space $L^1$ on $\partial\Omega$ is understood with respect to the Hausdorff measure (surface measure) $\cH^{d-1}$.
Moreover, since $p>d$ and $\Omega\subset \RR^d$ is Lipschitz, $W^{1,p}(\Omega)$ is compactly embedded into $C(\bar\Omega)$. Due to this compact embedding, 
any nonlinear force terms that are continuous on $C(\bar\Omega;\RR^d)$ or $C(\partial\Omega;\RR^d)$ are allowed as well, 
like
\begin{align*}
	&\int_\Omega G_{\rm body}(x,y)\,dx,\quad \text{with a $G_{\rm body}\in C(\bar\Omega\times \RR^d)$, or}\\
	&\int_{\partial\Omega} G_{\rm surface}(x,y)\,d\cH^{d-1}(x),\quad \text{with a $G_{\rm surface}\in C(\partial\Omega\times \RR^d)$}.
\end{align*} 
Finally, we could exploit the added regularity in form of terms that are weakly continuous in $W^{2,s}$, which allows even bulk and boundary terms involving $\nabla y$.
\end{rem}
\begin{rem}[boundary conditions] \label{rem:bc}
We also did not add any explicit boundary condition so far. Still, a weak form of a natural Neumann type boundary condition 
with the outer normal $\nu$ to $\partial\Omega$
is built in on all pieces of the boudary $\Lambda_N\subset \partial\Omega$ 
where $y$ is not subject to explicit other boundary conditions (if any), e.g.:
\[
	\frac{\sigma}{s}\abs{D^2y}^{s-2}D^2y:(\nu\otimes \nu) + D_F W(x,\nabla y)\cdot\nu+g_{surface}=0\quad\text{on $\Lambda_N$}.
\]
Here, we assumed that exactly one surface term was added to the energy, namely \eqref{forces-surface}.
Dirichlet conditions on $\Lambda_D$, the rest of the boundary, could be added. The limit of $E^{CN}_{\eps_2}$ is not directly affected by that
since the results of Section~\ref{sec:cn} obviously also hold for any restricted class of states.
Still, extra efforts in the proof of Theorem~\ref{thm:convergence} (ii) below would be required to make sure that the Dirichlet condition is always respected when we manipulate states. 
The extra requirements for the boundary data that would be needed then are the following:
If we impose
\[
	y=y_0~~\text{on $\Lambda_D$},~~\text{with $\Lambda_D\subset \partial \Omega$ relatively open,}
\]
the given boundary data $y_0:\Lambda_D\to \RR^d$ must have an extension to a state $y_0\in W^{2,s}(\Omega;\RR^d)$ 
which is far enough from any self-penetration so that
\begin{enumerate}
\item[(i)] $E_\sigma(y_0)<+\infty$;
\item[(ii)] $E^{CN}_{\eps_2}(y_0)\to 0$ as $\eps_2\to 0$. 
\end{enumerate}
In particular, $y_0$ must satisfy the Ciarlet-Ne\v{c}as condition \eqref{ciarletnecas}, and if $\beta>d$ (recall that $\beta$ is the parameter formally governing the blow-up rate of $E^{CN}_{\eps_2}$), $y_0$ must not have self-contact on the boundary, cf.~Corollary~\ref{cor:invertibility}.
\end{rem}
\begin{rem}
In their basic form without additional terms, $E_{\eps,\sigma}$ and $E_{\sigma}$ are translation invariant, i.e.,
constant vectors can be added to $y$ without changing the energy.
In particular, $E_{\eps,\sigma}$ and $E_{\sigma}$ are only coercive when these constants are removed.
This can be easily achieved by working in the quotient space $W^{2,s}(\Omega;\RR^d)/\RR^d$ or $W^{1,p}(\Omega;\RR^d)/\RR^d$. Alternatively, if translation invariance is broken by 
boundary conditions or additional terms in the energy, it suffices if these somehow fix the constant (e.g., by a Dirichlet condition) or control it (e.g., by a coercive nonlinear force term).
\end{rem}
For fixed $\eps$ and $h$, we always have the existence of an energy minimizer:
\begin{prop}[existence of minimizers]
Let $\sigma>0$ be fixed and suppose that \eqref{W0}--\eqref{W2}, \eqref{Weps} and \eqref{pqsd} hold. Then for every fixed  
$h,\eps_1,\eps_2>0$,
$E_{\eps,\sigma}$ and $E^h_{\eps,\sigma}$ attain their minima in $W^{2,s}$ and $Y_h\subset W^{2,s}$, respectively.
\end{prop}
\begin{proof}
All three summands of $E_{\eps,\sigma}$ (or $E^h_{\eps,\sigma}$) are weakly lower semicontinuous in $W^{2,s}$:
$E^{reg}$ is convex and thus weakly lower semicontinuous.
The other two terms are even weakly continuous, because $W^{2,s}$ (or its closed subspace $Y_h$) is compactly embedded in $W^{1,\infty}$ and $L^\infty$,
$y\mapsto \int_\Omega W_{\eps_1}(x,\nabla y)\,dx$ is continuous in $W^{1,\infty}$
and $y\mapsto E^{CN}_{\eps_2}(y)$ is continuous in $L^{\infty}$. 
Due to the definition of $E^{reg}$ and the lower bounds for $W$ and $W_\eps$, $E_{\eps,\sigma}$ and $E^h_{\eps,\sigma}$ are also coercive 
with respect to the semi-norm $\norm{y}:=\norm{D^2 y}_{L^s}+\norm{D y}_{L^p}$ on
$W^{2,s}$, which by Poincar\'e's inequality 
is a norm on the quotient space $W^{2,s}/\RR^d$ where functions differing only up to an additive constant vector are considered equivalent. 
(The quotient space is only needed when the translation invariance of the energy is not broken by boundary conditions or force terms.)
Hence, we get the existence of minimizers by the direct method of the calculus of variations.
\end{proof}
Our main results provides convergence of $E^h_{\eps,\sigma}$
and its minimum as $(h,\eps)\to 0$:
\begin{thm}\label{thm:convergence}
Let $\sigma>0$ be fixed and suppose that \eqref{W0}--\eqref{W2}, \eqref{Weps} and \eqref{pqsd} hold.
Then 
for every $(h(k),\eps(k))=(h(k),\eps_1(k),\eps_2(k))\in (0,\infty)^3$, 
$k\in\NN$, with 
$h(k)\to 0$, $\eps_1(k)\to 0$ and $\eps_2(k)\to 0$ as $k\to \infty$, we have the following two properties for all
$y\in W^{2,s}(\Omega;\RR^d)$:
\begin{enumerate}
\item[(i)] For every sequence $y_k\rightharpoonup y$ in $W^{2,s}$ (weakly),
\[
	\liminf_{k\to\infty} E_{\eps(k),\sigma}^{h(k)}(y_k)\geq E_{\sigma}(y);
\]
\item[(ii)] there exists a sequence $y_k\to y$ in $W^{2,s}$ (strongly) such that
\[
	\lim_{k\to\infty} E_{\eps(k),\sigma}^{h(k)}(y_k)= E_{\sigma}(y).
\]
\end{enumerate}
This also remains true for the case $h=0$ 
if we define $E^0_{\eps,\sigma}:=E_{\eps,\sigma}$.
\end{thm}
\begin{rem}[$\Gamma$-convergence]
If (ii) is slightly weakened to
\begin{enumerate}
\item[(ii)'] there exists a sequence $y_k\rightharpoonup y$ in $W^{2,s}$ (weakly) such that
\[
	\lim_{k\to\infty} E_{\eps(k),\sigma}^{h(k)}(y_k)= E_{\sigma}(y),
\]
\end{enumerate}
then (i) and (ii)' are exactly the definition of
$\Gamma(W^{2,s}\text{-weak})$-convergence of $E^h_{\eps,\sigma}$ to $E_{\sigma}$ as $(h,\eps)\to 0$, i.e., 
$\Gamma$-convergence with respect to the weak topology in $W^{2,s}$.
\end{rem}
\begin{rem}[convergence of minimizers]\label{rem:minimizersconverge}
For $\sigma>0$ fixed, 
the family of functionals $(E^h_{\eps,\sigma})_{(\eps,h)}$ is equi-coercive in $W^{2,s}/\RR^d$ 
due to \eqref{Weps} and 
the obvious properties of $E^{reg}_{\sigma}$.
As a consequence, $\Gamma$-convergence automatically implies that up to a subsequence (and subtracting suitable constants if necessary),
minimizers of $E^h_{\eps,\sigma}$ weakly converge in $W^{2,s}$ to a minimizer of the limit functional $E_{\sigma}$. 
Moreover,
$\Gamma(W^{2,s}\text{-weak})$-convergence of $E^h_{\eps,\sigma}$ to $E_{\sigma}$
is equivalent to $\Gamma(W^{1,p}\text{-weak})$-convergence.
\end{rem}
\begin{rem}\label{rem:noscalingregime}
Unlike the corresponding result in \cite{MieRou16a}, we do not 
need to prescribe any ``stability criterion'' linking $\eps_1$ and $h$. This is essentially a consequence of 
Lemma~\ref{lem:detbound} below that was slightly improved compared to its predecessor in \cite{HeaKroe09a}, see also Remark
\ref{rem:newsinceHK}. 
However, if 
$E^{CN}_{\eps_2}$ or other integrals in the energy are not computed exactly as assumed within this section, but only approximated by numerical integration (even in $Y_h$), 
we typically need $h$ of the order of $\eps_2$ or smaller to keep the approximation error on a tolerable level, cf.~Remark~\ref{rem:aura}.
\end{rem}
\begin{rem}\label{rem:PalmerHealey}
In \cite{PaHea17a} (also see \cite{Pa18a} for domains with Lipschitz boundaries), equilibrium equations for hyperelastic minimizers were derived. The Ciarlet-Ne\v{c}as condition there gives rise to a boundary force term with a force density given in form of a Radon measure concentrated on the self-contact set on the boundary (if any). In a sense, our penalty term in the limit as $\eps_2\to 0$ should represent an associated energy. However, a direct connection on the technical level is not quite obvious.
\end{rem}
For the proof of Theorem~\ref{thm:convergence}, we strongly rely on a result of \cite{HeaKroe09a} that yields a uniform positive lower bound $\det\nabla y\geq \delta>0$ for
all deformations with bounded energy, provided that the energy contains a higher order term 
that controls the norm of $J(x):=\det\nabla y(x)$ in a H\"older space.
This also uses that as a Lipschitz domain, $\Omega$ has an interior cone property: 
For each $x\in \Omega$, there exists a rotation $Q_x\in SO(d)$ such that $x+Q_xV\subset \Omega$, 
where
\[
  V:=B_\mu(0)\cap \{z=(z_1,\ldots,z_d)\in \RR^d\mid z_1 >\nu \abs{z}\}\subset \RR^d,
\]
is a fixed (cut-off) cone given by suitable constants $\mu>0$, $\nu<1$ independent of $x$.
For such domains, we have the following variant of \cite[Lemma 4.1]{HeaKroe09a}. Here, we also 
use slightly weaker assumptions and state additional explicit information about the constant $\delta$, 
but essentially, it is still based on the same ideas.
\begin{lem}\label{lem:detbound}
Suppose that $\Omega\subset \RR^d$ is bounded domain with an interior cone property as introduced above, and let 
$J\in C^{\alpha}(\Omega)$,
$\alpha\in(0,1)$.
In addition, suppose that
\begin{align}\label{Jconstants}
	\int_{\Omega} \max\{\delta,J(x)\}^{-q}\,dx\leq C
	~~~\text{and}~~~
	\sup_{\stackrel{x_1,x_2\in \Omega}{\abs{x_1-x_2}<\mu}} \frac{\abs{J(x_1)-J(x_2)}}{\abs{x_1-x_2}^\alpha}
 \leq M
\end{align}
where $q\geq d/\alpha$, $M>0$ are constants and
\[
  \delta:=\kappa^{-1}(C),~~~\kappa(t):=
	d\frac{\abs{V}}{\mu^d} \int_0^\mu \left(t+M \abs{r}^{\alpha}\right)^{-q}\,r^{d-1}\,dr,
	~~~t>0.
\]
Then $J(x)> \delta>0$ for all $x\in \Omega$.
\end{lem}
\begin{rem}\label{rem:newsinceHK}
Since $\delta$ depends on $C$, the first condition in \eqref{Jconstants} looks somewhat implicit. 
Typically, we a priori have it with some other constant instead of $\delta$, say, $\gamma\geq 0$. One can then compute $\delta$ 
(which does not depend on $\gamma$) and check a posteriori whether $\gamma\leq \delta$. If this is the case,
we automatically get \eqref{Jconstants} with $\delta$, too, because then
trivially $\int_{\Omega} \max\{\delta,J(x)\}^{-q}\,dx\leq \int_{\Omega} \max\{\gamma,J(x)\}^{-q}\,dx\leq C$.
We state \eqref{Jconstants} in this slightly complicated form 
to point out that the singular function $J\mapsto J^{-q}$ appearing there can 
be modified near the origin, removing the singularity, as long as one does not change the value for $J\geq \delta$. 
As we show in detail in Proposition~\ref{prop:yEbounded} below, this is quite useful because 
it means we can verify \eqref{Jconstants} for $J=\det\nabla y$
also using energy bounds for
our approximate elastic energy functionals involving the modified energy densities $W_{\eps_1}$, at least if $\eps_1$ is small enough.
\end{rem}
\begin{proof}[Proof of Lemma~\ref{lem:detbound}]
First notice that $\kappa$ is a strictly decreasing function with $\kappa(t)\to 0$ as $t\to+\infty$, and
$\kappa(t) \to +\infty$ as $t\searrow 0$ since 
$q\geq d/\alpha$. Hence, $\kappa^{-1}:(0,\infty)\to (0,\infty)$ is well defined and also strictly decreasing.
Moreover, while $\kappa$ was defined using polar coordinates, the integral also can be written in standard coordinates:
\begin{align}\label{galtdef}
  \kappa(t)=\int_{V}\left(t+M \abs{x}^{\alpha}\right)^{-q}\,dx
	=\int_{Q V}\left(t+M \abs{x}^{\alpha}\right)^{-q}\,dx,
\end{align}
for all $Q \in SO(d)$.
Now let $x_0\in\Omega$ and choose $Q=Q(x_0)\in SO(d)$ such that $x_0+Q V\subset \Omega$. 
Depending on $x_0$, we define $K\in C^\alpha(\Omega)$, 
\[
	K(x):=\left\{\begin{alignedat}[c]{2}
	&J(x) \quad&&\text{if $J(x_0)\geq \delta$,}\\
	&J(x)+\delta-J(x_0) \quad&&\text{if $J(x_0)<\delta$.}
	\end{alignedat}\right.
\]
Since $K\geq J$ and their difference is a constant, \eqref{Jconstants} implies that
\begin{align}\label{Jconstants2}
	\int_{\Omega} \max\{\delta,K(x)\}^{-q}\,dx\leq C
	~~~\text{and}~~~
	\sup_{\stackrel{x_1,x_2\in \Omega}{\abs{x_1-x_2}<\mu}} \frac{\abs{K(x_1)-K(x_2)}}{\abs{x_1-x_2}^\alpha}
 \leq M
\end{align}
In addition, $K(x_0)\geq \delta$, and from \eqref{Jconstants2} and \eqref{galtdef} we thus get that
\begin{align*}
\begin{aligned}
	\kappa(K(x_0))	
	&\leq \int_{\Omega \cap B_{\mu}(x_0)}(K(x_0)+M
	\abs{x-x_0}^{\alpha})^{-q}\,dx \\
	&= \int_{\Omega \cap B_{\mu}(x_0)}\max\{\delta,K(x_0)+M
	\abs{x-x_0}^{\alpha}\}^{-q}\,dx \\
	&<\int_\Omega \max\{\delta,K(x)\}^{-q}\,dx\leq C.
\end{aligned}
\end{align*}
Hence, $K(x_0)>\kappa^{-1}(C)=\delta$, and therefore $J(x_0)=K(x_0)>\delta$.
\end{proof}
We can now derive additional properties for deformations with bounded energy.
\begin{prop}\label{prop:yEbounded}
Let $\Omega\subset\RR^d$ be a bounded Lipschitz domain, let $\sigma>0$,
suppose that \eqref{W0}, \eqref{W1}, \eqref{Weps} and \eqref{pqsd} hold, and let $K>0$. 
Then there is a constant $\bar\eps_1>0$ such that 
for every $0<\eps_1\leq \bar\eps_1$ and every 
\[
	y\in \cB_K(\eps_1):=\mysetr{y\in W^{2,s}(\Omega;\RR^d)}{
	\textstyle E_\sigma^{reg}(y)+E^{el}_{\eps_1}(y)\leq K
	},
\]
\eqref{yconstants}
holds with $\alpha:=\frac{s-d}{s}$ and suitable constants $\delta>0$, $M_1,M_2\geq 0$ independent of $y$ and $\eps_1$.
That is,
\begin{align}\label{yconstantsB}
	\text{$\det \nabla y\geq \delta>0$}
	~~~\text{and}~~~
	\text{$\abs{\nabla y}\leq M_1$}~~~\text{on $\Omega$,}~~~\text{and}~~~\norm{\nabla y}_{C^\alpha(\Omega)}\leq M_2.
\end{align}
This also holds for $\eps_1=0$ if we replace $E^{el}_{\eps_1}$ by $E^{el}$.
\end{prop}
\begin{proof}
We only discuss the case $\eps_1>0$; 
the case $\eps_1=0$ is a similar and more straightforward application of Lemma~\ref{lem:detbound}.
Since $E_\sigma^{reg}(y)=\sigma\int_\Omega \abs{D^2y}^s\,dx$ and $W_{\eps_1}$ satisfies the lower bound stated in \eqref{Weps},
$y\in \cB_K(\eps_1)$ implies that $\norm{D^2y}_{L^s}^s+c_3 \norm{\nabla y}_{L^p}^p \leq K+c_4$. In particular,
$\nabla y$ is bounded in $W^{1,s}(\Omega;\RR^{d\times d})$, which is continuously embedded in $C^\alpha(\Omega;\RR^{d\times d})$ and
$C(\bar\Omega;\RR^{d\times d})$. 
Hence, we immediately get the last two inequalities in \eqref{yconstantsB}. It remains to show the lower bound for $\det \nabla y$.
This will be obtained by applying Lemma~\ref{lem:detbound}, and we therefore have to verify \eqref{Jconstants} 
for $J:=\det \nabla y$. Since $\norm{\nabla y}_{L^\infty}\leq M_1$ and $\norm{\nabla y}_{C^\alpha}\leq M_2$,
the H\"older semi-norm of $\det\nabla y$ (a polynomial of degree $d$ in $\nabla y$) appearing in \eqref{Jconstants} is bounded by 
\[
	M:=d M_1^{d-1} M_2. 
\]
For a proof of the first inequality in \eqref{Jconstants}, let $\gamma>0$.
For $y\in \cB_K(\eps_1)$, we have that $\int_\Omega W_{\eps_1}(x,\nabla y)\,dx\leq K$, 
and as a consequence 
of this and \eqref{Weps}, we obtain the following estimate for all $\eps_1\leq \gamma$:
\[
\bald
  \int_\Omega \max\{\gamma,\det \nabla y(x)\}^{-q}\,dx 
	\leq	
	\frac{c_4}{c_3}\abs{\Omega}+\frac{1}{c_3}\int_{\Omega} W_{\eps_1}(x,\det \nabla y)\,dx 	
	\leq C &\\
	\text{with}~~C:=\frac{c_4}{c_3}\abs{\Omega} +\frac{1}{c_3} K&.
\eald
\]
Notice that $C$ does not depend on $\gamma$ or $\eps_1$.
Hence, we may use $\bar\eps_1:=\gamma:=\delta$,
with the constant $\delta$ of Lemma~\ref{lem:detbound} (with $C$ and $M$ as defined above).
The lemma then entails that $J=\det\nabla y\geq \delta>0$, i.e.,
the first inequality in \eqref{yconstantsB}.
\end{proof}

The final missing ingredients for the proof of Theorem~\ref{thm:convergence} are some uniform continuity 
properties of the terms in the energy.
\begin{prop}\label{prop:SGucont}
Let $\Omega\subset\RR^d$ be a bounded domain, $\sigma>0$ and $s\geq 1$.
Then 
$y\mapsto E^{reg}_\sigma(y)=\sigma \int_\Omega |D^2 y(x)|^s\,dx$ 
is uniformly (Lipschitz) continuous on all bounded subset of $W^{2,s}(\Omega;\RR^d)$.
\end{prop}
\begin{proof}
This is a simple consequence of H\"older's inequality and the elementary $s$-Lipschitz continuity of $G\mapsto \abs{G}^s$,
$\RR^{d\times d\times d}\to \RR$:\\
$\abs{\abs{G_1}^s-\abs{G_2}^s}\leq s (\abs{G_1}^{s-1}+\abs{G_2}^{s-1}) \abs{G_1-G_2}$.
\end{proof}
\begin{prop}\label{prop:Wucont}
Let $\Omega\subset\RR^d$ be a bounded Lipschitz domain,
suppose that \eqref{W0}, \eqref{W1}, \eqref{Weps} and \eqref{pqsd} hold, and let $K>0$. 
Then there exists a constant $\bar\eps_1>0$ 
and a modulus of continuity $\theta$ (i.e., $\theta:[0,\infty)\to [0,\infty)$ continuous and increasing with $\theta(0)=0$)
such that 
for every $0<\eps_1 \leq \bar\eps_1$, 
\begin{align}\label{pWuc-1}
\bald
  \abs{E^{el}_{\eps_1}(y_1)-E^{el}(y_2)\,dx} 
	\leq \abs{\Omega}\eps_1+\theta(\norm{y_1-y_2}_{W^{1,\infty}})&\\
	\qquad \text{for all $y_1,y_2\in \cB_K(\eps_1)$}&,
\eald
\end{align}
where $\cB_K(\eps_1)$ is the set defined in Proposition~\ref{prop:yEbounded}.
We emphasize that $\theta$ is independent of $y_1$, $y_2$ and $\eps_1$.
\end{prop}
\begin{proof} By \eqref{Weps}, 
\[
	\abs{W_{\eps_1}(x,F)-W(x,F)}\leq \eps_1\quad\text{if }\abs{F}\leq \frac{1}{\eps_1}~~\text{and}~~\det F\geq \eps_1.
\]	
In view of 
Proposition~\ref{prop:yEbounded}, it therefore suffices to show 
that for a suitable modulus of continuity $\theta$,
\begin{align}\label{pWuc-2}
\bald
  &\abs{\int_\Omega W(x,\nabla y_1(x))\,dx
	-\int_\Omega W(x,\nabla y_2(x))\,dx} \\
	&\qquad \leq \theta(\norm{\nabla y_1-\nabla y_2}_{W^{1,\infty}})
	\qquad \text{for all $y_1,y_2\in \tilde\cB_K$},
\eald
\end{align}
i.e., the uniform continuity of
$y\mapsto \int_\Omega W(x,\nabla y)\,dx$ 
with respect to the topology of $W^{1,\infty}$ on
the set
\[
	\tilde\cB_K:=\mysetr{y\in W^{2,s}(\Omega;\RR^d)}{
	\abs{\nabla y}\leq M_1~~\text{and}~~ \det \nabla y\geq \delta ~~\text{in $\Omega$}
	}.
\]
If $W$ does not depend on $x$, \eqref{pWuc-2} is obvious as for each $x$, $W(x,\cdot):\RR^{d\times d}\to \RR\cup \{+\infty\}$ is continuous and therefore uniformly continuous 
on any compact set where it is finite. For a general Carath\'eodory function $W$,
\eqref{pWuc-2} still follows from similar reasoning, as a consequence of the Scorza-Dragoni theorem (continuity of $W$ on compact sets with complements 
of arbitrarily small measure in $\Omega$, see \cite{Da08B}, e.g.).
\end{proof}

\begin{proof}[Proof of Theorem~\ref{thm:convergence}]
We will only provide a proof for the case involving Galerkin approximations with $h(n)>0$, $h(n)\to 0$. The case $h=0$ is similar and slightly simpler.

{\bf (i) ``Lower bound'':} Let $y_n\rightharpoonup y$ as $n\to\infty$, weakly in $W^{2,s}$. By compact embedding, this implies that
$y_n\to y$ strongly in $W^{1,\infty}$.
Passing to a suitable subsequence (not relabeled), we may assume that
$e_0:=\liminf E^{h(n)}_{\sigma,\eps(n)}(y_n)=\lim E^{h(n)}_{\sigma,\eps(n)}(y_n)$. In addition, we may assume that 
$e_0<+\infty$ because otherwise there is nothing to show. With $K:=e_0+1$,
we have $E^{h(n)}_{\sigma,\eps(n)}(y_n)\leq K$ for all $n$ sufficiently large.
Since $E^{CN}_{\eps_2(n)}\geq 0$, we infer that $y_n\in \cB_K(\eps_1(n))$ 
for all such $n$, where $\cB_K(\eps_1(n))$ is the set introduced in Proposition~\ref{prop:yEbounded}.
Due to 
Proposition~\ref{prop:Wucont}, we therefore have that
\begin{align}\label{tconv-1}
	\lim_{n\to\infty} \int_\Omega W_{\eps_1(n)}(x,\nabla y_n)\,dx = \int_\Omega W(x,\nabla y)\,dx.
\end{align}
Moreover, by weak lower semicontinuity of the convex functional $E^{reg}_\sigma$,
\begin{align}\label{tconv-2}
	\liminf_{n\to\infty} E^{reg}_\sigma(y_n)\geq E^{reg}_\sigma(y).
\end{align}
Besides \eqref{tconv-1} and \eqref{tconv-2}, 
it also trivially holds that $E^{CN}_{\eps_2(n)}\geq 0$. We thus get
$\liminf E^{h(n)}_{\sigma,\eps(n)}(y_n)\geq E_\sigma(y)$ as asserted, provided that
$y$ satisfies the Ciarlet-Ne\v{c}as condition.
For the proof of the latter, first observe that due to Proposition~\ref{prop:yEbounded}, \eqref{yconstants} is satisfied  
for each $y_n$ (in place of $y$), and Theorem~\ref{thm:cnsoft} 
is applicable, at least as long as $n$ is also large enough so that $\eps_2(n)\leq \bar\eps$.
We also know that
$E^{CN}_{\eps_2(n)}(y_n)$ is bounded from above because $e_0<+\infty$ and the other terms in $E^{h(n)}_{\sigma,\eps(n)}(y_n)$ are bounded from below.
Since $y_n\to y$ in $L^\infty$, Proposition~\ref{prop:CNstab} and Remark~\ref{rem:CNstab} therefore yield that
$\{N_y\circ y\}=P_y(0)=\emptyset$. In particular, $y$ is globally invertible and the Ciarlet-Ne\v{c}as condition holds for $y$.

{\bf (ii) Existence of a (strongly converging) ``recovery sequence'':}
Let $y\in W^{2,s}(\Omega;\RR^d)$. 
If $E_\sigma(y)=+\infty$,
any sequence $(y_n)\subset Y_{h(n)}$ with $y_n\to y$ in $W^{2,s}$ is suitable, in particular
$\lim E^{h(n)}_{\sigma,\eps(n)}(y_n)=+\infty=E_\sigma(y)$ as a consequence of (i).
We thus may assume that $E_\sigma(y)<+\infty$.
In particular, the Ciarlet-Ne\v{c}as condition 
holds for $y$ and
Proposition~\ref{prop:yEbounded} (for the case $h=0$) is applicable with $K:=E_\sigma(y)$, 
which yields \eqref{yconstantsB}.
Due to \eqref{yconstantsB}, Lemma~\ref{lem:biLi} can be applied. Hence, $y$ is also locally bi-Lipschitz, and
for such maps, the Ciarlet-Ne\v{c}as condition is equivalent to global invertibility in the classical sense.

Nevertheless, $y$ may still exhibit self-contact on the boundary. 
This is problematic for our construction because 
we might lose control of $E^{CN}_{\eps_2(n)}$. We therefore first artificially create a little gap around the boundary, 
with the ultimate goal of finding a suitable sequence $y_n\in Y_{(h(n))}$
approximating $y$ and its energy while $E^{CN}_{\eps_2}(y_n)=0$ for all $n$ (large enough). 
To create this gap,
let $\Psi_j:\Omega\to \Omega$ 
be a sequence of globally invertible maps of class $C^\infty$ which
``shrink'' $\Omega$ into a slightly smaller set and converge to the identity, more precisely,
\begin{align}\label{tconv-shrink}
\bald
  \Psi_j(\Omega)\subset \Omega^{(j)}:=\mysetr{x\in\Omega}{\dist{x}{\partial\Omega}>\frac{1}{j}},\\
	\norm{\Psi_j-id}_{W^{2,\infty}}\underset{j\to\infty}{\To} 0.	
\eald
\end{align}
Such maps $\Psi_j$ are easy to define locally in a neighborhood of a boundary point where $\partial\Omega$ can be represented as the graph of a Lipschitz function.
Globally, the local pieces can be glued together using a decomposition of unity; we omit the details.
As a consequence of \eqref{tconv-shrink}, 
\begin{align}\label{tconv-9}
	z_j\underset{j\to\infty}{\To} y~~\text{in $W^{2,s}$ (and $W^{1,\infty}$)}\quad \text{for}~z_j:=y\circ \Psi_{j}.
\end{align}
The function $z_j$, like other possible perturbations $z$ of $y$ in $W^{2,s}$, 
inherits \eqref{yconstantsB} with slightly modified constants:
There is a constant $\kappa>0$ such that
for all $z\in W^{2,s}\subset C^{1,\alpha} \subset W^{1,\infty}$ with $\norm{z-y}_{W^{2,s}}\leq \kappa$,
\begin{align}\label{tconv-10}
	\text{$\det \nabla z\geq \tilde{\delta}>0$}
	~~~\text{and}~~~
	\text{$\abs{\nabla z}\leq \tilde{M}_1$}~~~\text{on $\Omega$,} 
	~~~\text{and}~~~\norm{\nabla z}_{C^\alpha(\Omega)}\leq \tilde{M}_2,&
\end{align}
where $\tilde{\delta}:=\frac{1}{2}\delta$, $\tilde{M}_1:=M_1+1$ and $\tilde{M}_2:=M_2+1$.
In particular, $z=z_j$ is admissible if $j$ is sufficiently large.
Due to \eqref{tconv-9}, \eqref{tconv-10} and Proposition~\ref{prop:Wucont}, 
\begin{align}\label{tconv-12}
\bald
	\abs{E^{el}_{\eps_1(n)}(z)
	-E^{el}(y)}\leq \abs{\Omega}\eps_1(n)+\theta(c\norm{z-y}_{W^{2,s}}),&\\
	\quad\text{for all $z\in W^{2,s}$ with $\norm{z-y}_{W^{2,s}}\leq \kappa$.}&
\eald
\end{align}
Here, $c>0$ denotes the embedding constant 
such that $c\norm{z-y}_{W^{2,s}}\leq \norm{z-y}_{W^{1,\infty}}$,
and $\kappa>0$ is the constant governing the admissible functions in \eqref{tconv-10}.

Next, we claim that essentially due to the gap
around the boundary we have for any fixed $j$, there exists $n_0=n_0(j)$
such that for all $n\geq n_0$,
$E^{CN}_{\eps_2(n)}(z_j)=0$, and the same also holds for close enough
perturbations $z$ of $z_j$:
\begin{align}\label{tconv-15}
	E^{CN}_{\eps_2(n)}(z)=0\quad\text{for $n\geq n_0(j)$, $z$ with $\norm{z-z_j}_{W^{2,s}}<2\cE(h(n))$},
\end{align}
where $\cE(h(n))$ is the maximal Galerkin approximation error from \eqref{fespaceapprox}.
We choose $n_0=n_0(j)$ (w.l.o.g.~increasing) such that for all $n\geq n_0$,
\begin{align}\label{tconv-13}
	\frac{1}{j}\geq s(n):=\frac{2M_1^{d-1}}{\delta}t(n),\quad t(n):=R \eps_2(n)+2\cE(h(n)). 
\end{align}
and $s(n)\leq \varrho$. Here, $R:=\diam(\Omega)$, $\delta>0$, $M_1>0$ are the constants from \eqref{yconstantsB}, and $\varrho$ 
is the radius of local invertibility from 
Lemma~\ref{lem:biLi}. 
We claim that for any such $n$, 
all $z$ with $\norm{z-z_j}_{W^{2,s}}<2\cE(h(n))$ 
and any pair $x_1,x_2\in\Omega$ with $\abs{x_1-x_2}>\frac{\varrho}{2}$,
\begin{align}\label{tconv-14}
	\abs{z(x_1)-z(x_2)}\geq \abs{z_j(x_1)-z_j(x_2)}-2\cE(h(n))
	\geq	R \eps_2(n).
\end{align}
Given \eqref{tconv-14}, Corollary~\ref{cor:whenECNvanishes} immediately implies \eqref{tconv-15}.
We prove \eqref{tconv-14} indirectly (the second inequality; the first follows from the triangle inequality). Suppose that 
\begin{align}\label{tconv-16}
	\abs{z_j(x_1)-z_j(x_2)}	<	t(n)=R \eps_2(n)+2\cE(h(n)).
\end{align}
Since $\Psi_j(x_1)\in \Omega_j\subset\Omega$ 
(and $\Psi_j(x_2)$ likewise), we have 
that
\begin{align}\label{tconv-17}
	\dist{\Psi_j(x_1)}{\partial\Omega} >\frac{1}{j} \geq s(n) 
\end{align}
by \eqref{tconv-shrink} and \eqref{tconv-13}.
But on the other hand, if \eqref{tconv-17} holds, then
$B_{s(n)}(\Psi_j(x_1))\subset \Omega$. Since $y$ is bi-Lipschitz on $B_{s(n)}(\Psi_j(x_1))$ 
($s(n)\leq \varrho$; also notice that the constant factor linking $s(n)$ and $t(n)$ is exactly the constant from the lower bound in \eqref{lemIMT-biLi}),
we infer from \eqref{tconv-16} that
\[
	z_j(x_2)\in B_{t(n)}(z_j(x_1))=B_{t(n)}(y(\Psi_j(x_1))) \subset y(B_{s(n)}(\Psi_j(x_1))).
\]
This is impossible because $z_j(x_2)=y(\Psi_j(x_2))$ and
$y$ and $\Psi_j$ are injective, which concludes the proof of \eqref{tconv-15}.

In particular, we may use $z=z^{(h(n))}_j$ in \eqref{tconv-15}
with a sufficiently close Galerkin approximation $z^{(h(n))}_j\in Y_{h(n)}$ 
of $z_j$. Now take $j(n)\to\infty$ as $n\to\infty$, but slow enough so that $n_0(j(n))\leq n$.
With this choice, 
\[
	\norm{y_n-z_{j(n)}}_{W^{2,s}}<2\cE(h(n))\underset{n\to\infty}{\to}0\quad \text{for}~y_n:=z^{(h(n))}_{j(n)}\in Y_{h(n)}
\]
which together with \eqref{tconv-9} implies that $y_n\to y$ in $W^{2,s}$. 
By Proposition~\ref{prop:SGucont},
we see that
\begin{align}\label{tconv-21}
	E^{reg}_\sigma(y_n)
	\to E^{reg}_\sigma(y)
	\quad\text{as $n\to\infty$},
\end{align}
and from \eqref{tconv-12}, we infer that
\begin{align}\label{tconv-22}
	E^{el}_{\eps_1(n)}(y_n)
	\to E^{el}(y)
	\quad\text{as $n\to\infty$}.
\end{align}
Finally, \eqref{tconv-15} yields that
\begin{align}\label{tconv-23}
	E^{CN}_{\eps_2(n)}(y_n)=0\quad\text{for all $n$ large enough}
\end{align}
(more precisely, $n$ large enough so that \eqref{tconv-10} holds 
for $z=z_{j(n)}$ and $z=y_n$).
Altogether, $E_{\sigma,\eps(n)}^{h(n)}(y_n)\to E_\sigma(y)$ as asserted.
\end{proof}

\section{Numerical experiments \label{sec:numerics}}


We consider $d=2$ 
and the approximate deformation $$y = (y_1, y_2) \in W^{2,s}(\Omega;\RR^2)$$ is searched for as the (ideally global) minimizer of 
\begin{align}
	E_{\eps,\sigma,\mu}(y)=E^{el}_{\eps_1}(y)+\mu E^{CN}_{\eps_2}(y)+E^{reg}_{\sigma}(y)+E^{body}(y) \label{energy:total}
\end{align}
over a finite dimensional space. 
Both deformation components $y_1, y_2 $ are discretized in the space of the Bogner-Fox-Schmit (BFS) rectangular elements \cite{BoFoSc65} that provide continuous differentiability of approximations. 
Here,
\[
	E^{body}(y):=\int_\Omega g_{\rm body}(x)\cdot y(x)\,dx
\]
is the energy contribution of a body force of type \eqref{forces-body}, with
$g_{\rm body}$ as specified below (in Model II; $g_{\rm body}=0$ in Model I). We here
fix the constants
\[	
	\eps_1:=\frac{1}{100},~~\sigma:=1,~~s:=4,~~p:=4~~\text{and} ~~q:=6
\]
(in particular, \eqref{pqsd} is satisfied for $d=2$).
The Ciarlet-Ne\v{c}as penalty term $E^{CN}_{\eps_2}$
is included with a weight $\mu$ for which we tested several values between $0$ and $1$; notice that $\mu=0$ completely switches off the penalty term.
As before, $E^{CN}_{\eps_2}$ is given by \eqref{sCN1}, and 
for the computational tests, we chose $g(t)=t$,
$\beta=1.8$ or $\beta=2.2$ (two values close to $d=2$, the threshold for Corollary~\ref{cor:invertibility}),
and also experiment with different values for $\eps_2$, typically adapted to the grid size $h$.
As an example for the higher order term, we employ
\[
	E^{reg}_{\sigma}(y)=\sigma \int_\Omega |D^2 y(x)|^s\,dx.
\]
Finally, we use the elastic part of the energy
$$E^{el}_{\eps_1}(y)=\int_\Omega W^{el}_{\eps_1}(\nabla y(x))\,dx$$
with density chosen as follows, for all $F\in \RR^{2\times 2}$ and $J:=\det F$:
\[
\begin{aligned}
	W^{el}_{\eps_1}(F):=&|F|^{p}-d^{\frac{p}{2}}\\
	&-\frac{p}{q}d^{\frac{p}{2}-1}+\frac{p}{q} d^{\frac{p}{2}-1}
	\left\{
	\begin{alignedat}{2}
	&J^{-q} &&~~\text{if $J\geq \eps_1$},\\
	&-q\eps_1^{-q-1}(J-\eps_1)+\eps_1^{-q} &&~~\text{if $J<\eps_1$}.
	\end{alignedat}
	\right.
	\end{aligned}
\] 
Here, recall that $d=2$, although the example could also be used for higher dimensions. 
Above, 
$$\nabla y(x)\in \RR^{2\times 2}, \quad D^2 y(x)\in \RR^{2\times 2 \times 2}$$ and denote the gradient and the Hessian of 
$y:\Omega\to \RR^2$, respectively,
and the norms $\abs{\cdot}$ are euclidean (Frobenius): $\abs{F}:=\big(\sum_{i,j} F_{ij}^2\big)^{\frac{1}{2}}$ for $F=(F_{ij})\in \RR^{2\times 2}$,
$\abs{G}:=\big(\sum_{i,j,k} G_{ijk}^2\big)^{\frac{1}{2}}$ for $G=(G_{ijk})\in \RR^{2\times 2\times 2}$. 
\begin{rem}
$W^{el}_{\eps_1}$ is polyconvex and frame indifferent. Moreover, if $0<\eps_1<1$,
$W^{el}_{\eps_1}(F)\geq 0$ with equality if and only if $F$ is a rotation matrix.
The second part of $W^{el}_{\eps_1}$ is a $C^1$-function in $J=\det F$, and the two cases define a truncated version of $J\mapsto J^{-q}$ using an affine extension for $J<\eps_1$.
For the shapes, body forces and boundary conditions used in our examples, there is no incentive for the material to create spots with high local compression. As a consequence, 
the results are independent of the choice of $\eps_1<<1$, 
as the computed deformations stay far away from the regime $\det \nabla y<\eps_1$ anyway (as the optimal deformations are expected to), 
for any reasonably small choice of $\eps_1$. 
\end{rem}
\begin{rem}
For the actual computations, we have replaced the non-differentiable functions $[\cdot]^+$ and $g(\abs{\cdot})$ appearing in the definition of $E^{CN}_{\eps_2}$ by 
$C^1$-approximations $h([\cdot]^+)$ and $h(g(\abs{\cdot}))$,  where 
\[
  h(x) : = \left.
  \begin{cases}
    0, & \text{for }  x \leq 0, \\
    x^2/(2a) & \text{for } 0 \leq x \leq a, \\
    x-a/2, & \text{for } a \leq x 
  \end{cases}
  \right. 
	\]
	for some small parameter $a>0$ (we set $a=1/10$ in all computations). This smoothing is introduced in order to avoid the risk that the actual minimizer sits at a point where the functional does not have a well defined derivative. The latter 
could cause serious problems for the solver. While changing $g$ is fully covered by the theory developed above and can at most affect constants appearing in the theoretical results, changing $[\cdot]^+$ even slightly around zero can potentially effect the scaling of $E^{CN}_{\eps_2}$ as $\eps_2\to 0$. 
But in practice the asymptotics as $\eps_2\to 0$ cannot easily be observed numerically anyway.
\end{rem}

\subsection{Model I}
We consider a reference configuration $\Omega = \Omega_1 \, \cup \, \Omega_2 \subset \RR^2$ which consists of two rectangular boxes 
$\Omega_1=(0,2) \times (0.5, 1.5)$ and $\Omega_2=(0,2) \times (-1.5, 0.5)$. See Figure \ref{fig:model1} for illustration. 
\begin{figure}[h]
    \begin{minipage}[c]{.32\textwidth}
        \centering
	\includegraphics[width=\textwidth]{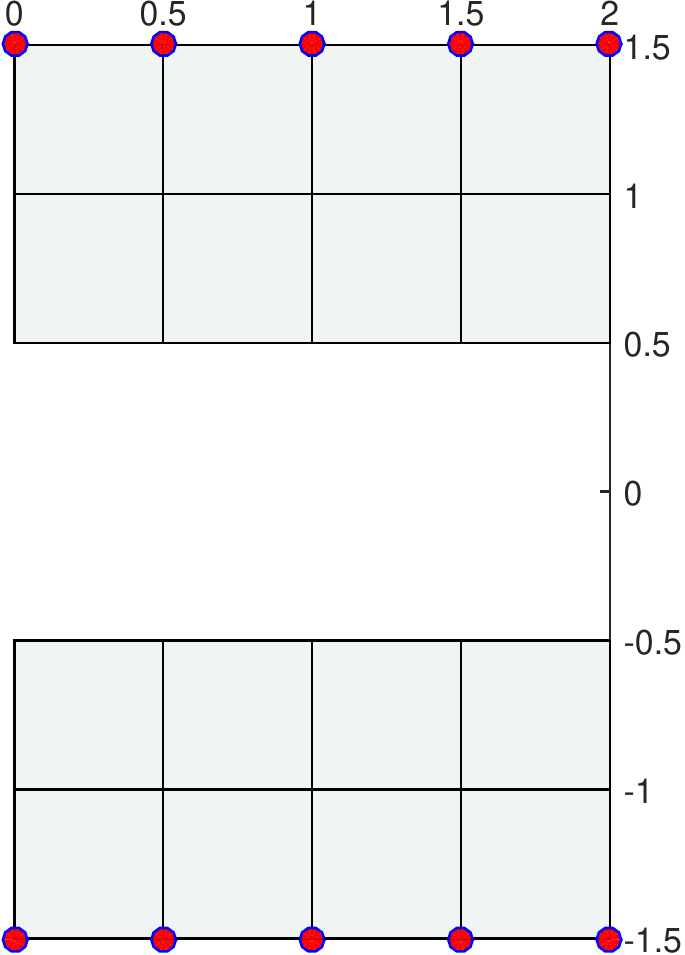}    
    \end{minipage}
   \caption{Model I : Coarse rectangular mesh with Dirichlet boundary nodes indicated by full  dots.}
   \label{fig:model1}
\end{figure}
We impose nonhomogeneous Dirichlet boundary conditions on a part of the boundary:
\begin{eqnarray*}
& y_1=x_1 + m_1, y_2 = x_2 - m_2 \quad \mbox{for } (x_1, x_2) \in \Lambda_{D,1}, \\
& y_1=x_1, \qquad  y_2 = x_2 + m_2 \quad \mbox{for } (x_1, x_2) \in \Lambda_{D,2},
\end{eqnarray*}
where $\Lambda_{D,1} :=(0,2) \times \{ 1.5 \}$ and $\Lambda_{D,1} :=(0,2) \times \{ -1.5 \}$ and $m_1, m_2 \in \RR$ are parameters. There is no linear body force term considered in this model. We consider a sequence of minimization problems with parameters
$$ m_2 \in \{ 0.4, 0.5, 0.6, 0.7 \} $$
and the same parameters 
$m_1= 0.2$ 
and $\mu =1$ in two different cases
$$\epsilon_2=1/4, \quad \epsilon_2=1/8.$$

Figure \ref{fig:energies_vs_m2}  displays how much the total energy $E_{\eps,\sigma,\mu}(y)$ and the scaled penetration energy $E^{CN}_{\eps_2}(y)$ depend on the parameter $m_2$. Unsurprisingly, both the energy and the influence of $E^{CN}_{\eps_2}(y)$ grow the larger $m_2$ gets, i.e., the more the two pieces are pushed against each other by the boundary conditions.
Since there is no linear body force term considered in this model, the difference energy $E_{\eps,\sigma,\mu}(y) - E^{CN}_{\eps_2}(y)$ converts to $E^{el}_{\eps_1}(y)$ and $E^{reg}_{\sigma}(y)$. Figure \ref{fig:model1_optimization} shows a few resulting optimized domains. 
Choosing $\eps_2=1/8$ is more or less the smallest reasonable choice for $\eps_2$ given the grid size (cf.~Remark~\ref{rem:aura}). In this case, one can already see effects of errors due to the the numerical integration in the marginal density of $E^{CN}_{\eps_2}$, which is much more uniform and intuitive for $\eps_2=1/4$.

\begin{figure}[H]
\begin{minipage}[t]{.49\textwidth}
\centering
\includegraphics[width=\textwidth]{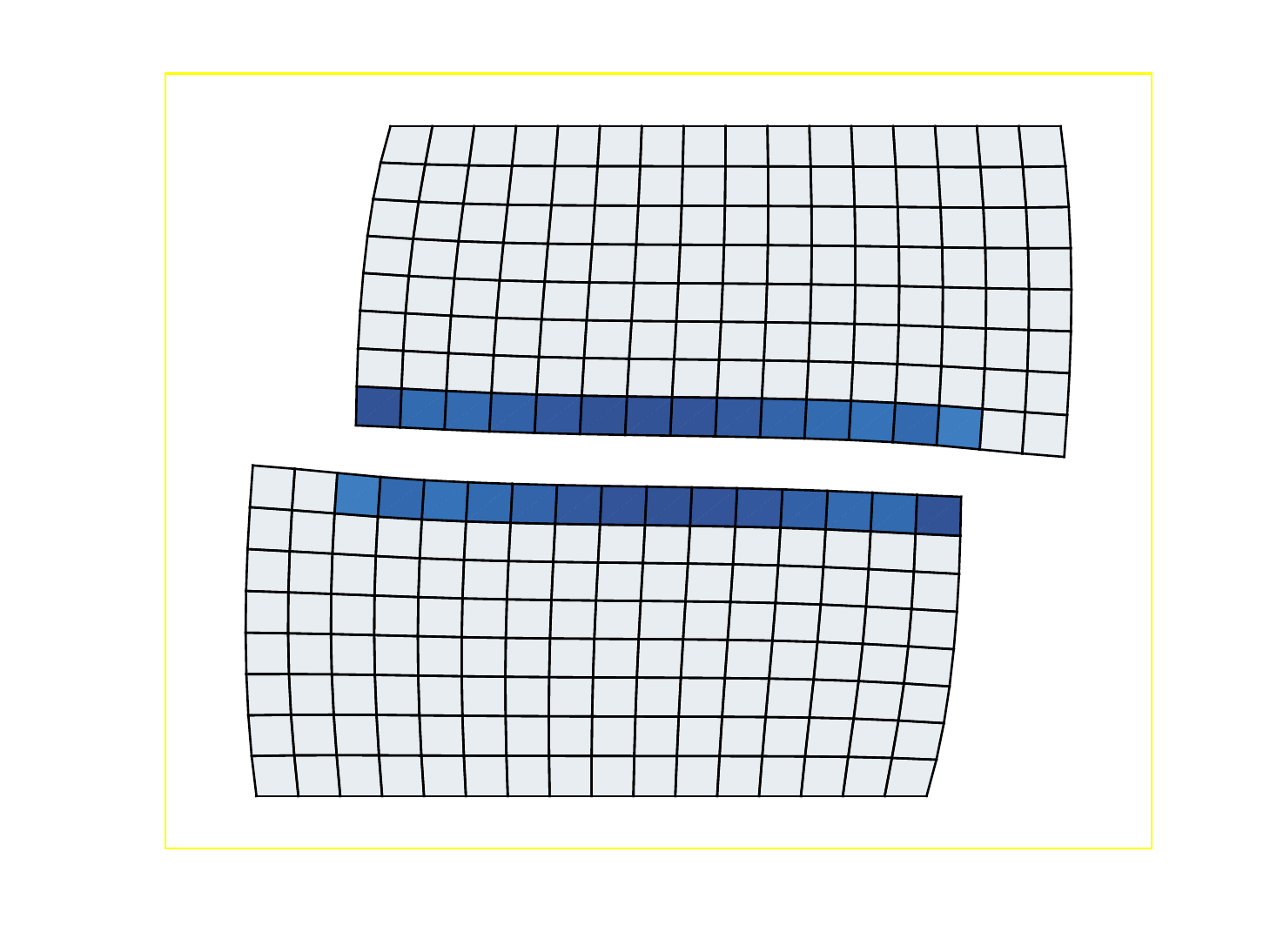}
\subcaption{$m_2=0.5, \epsilon_2=1/4$.}
\includegraphics[width=\textwidth]{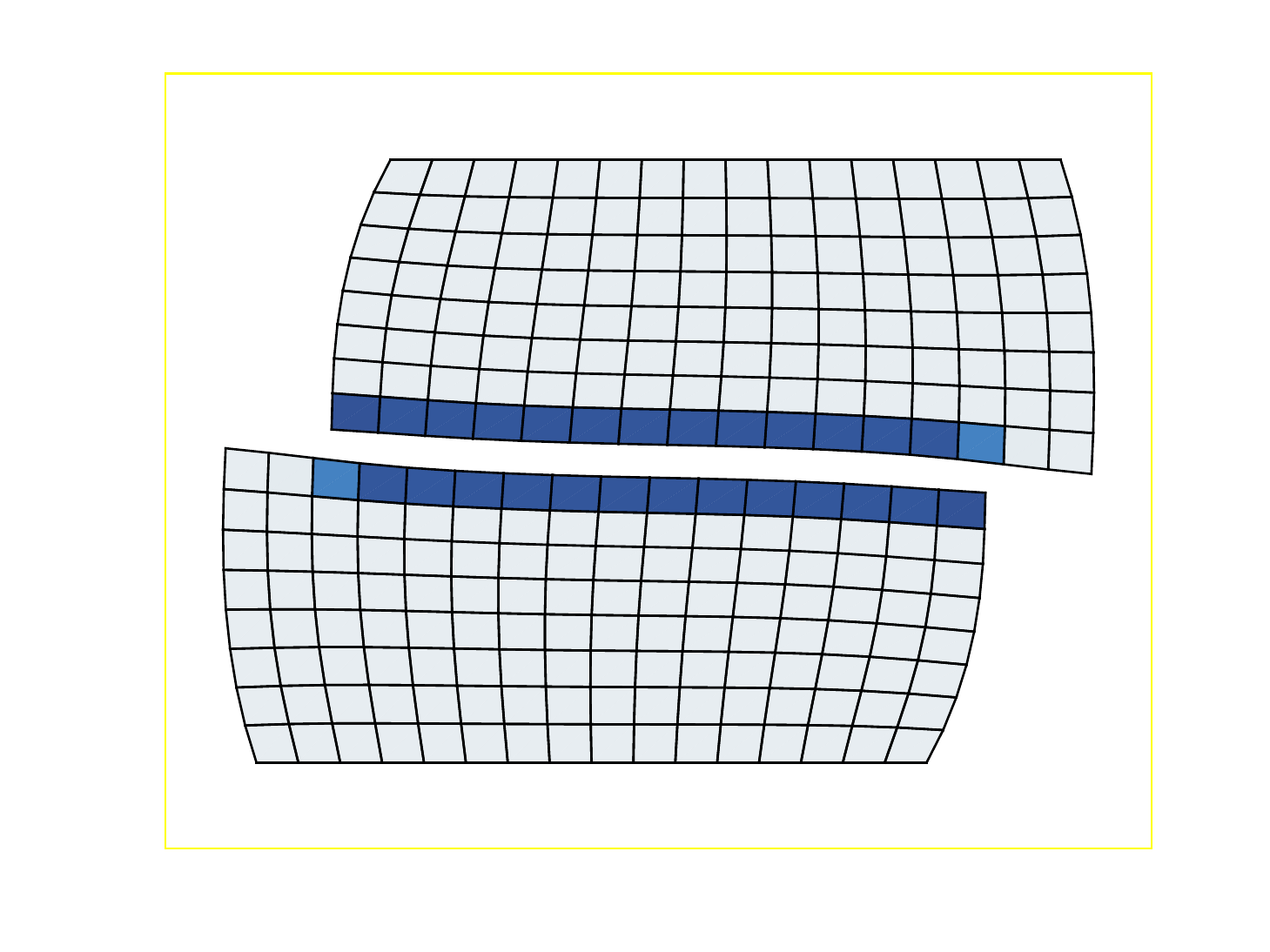}     
\subcaption{$m_2=0.6,  \epsilon_2=1/4$.}		
\includegraphics[width=\textwidth]{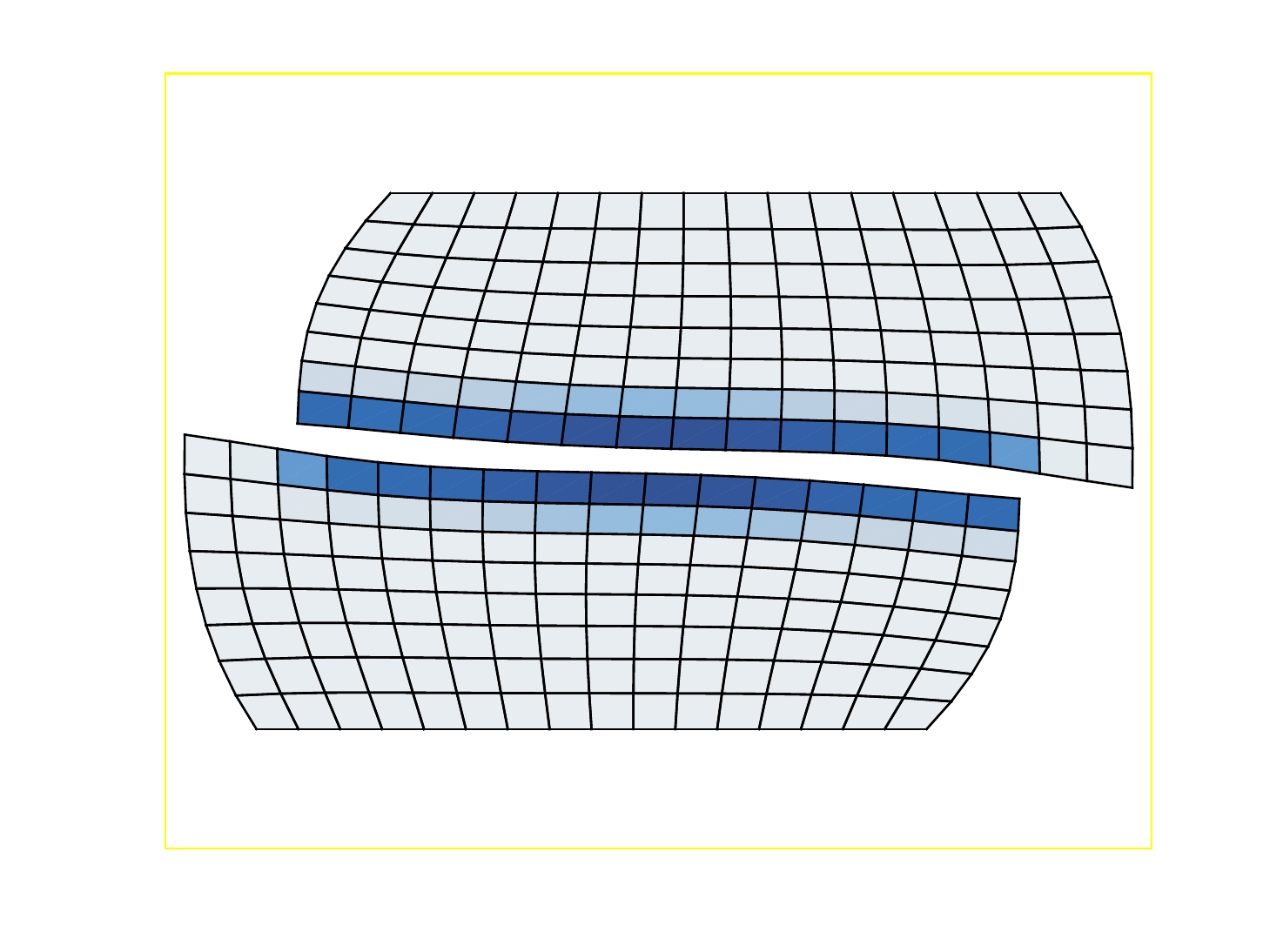}     
\subcaption{$m_2=0.7,  \epsilon_2=1/4$.}
\end{minipage}  
\begin{minipage}[t]{.49\textwidth}
\centering
\includegraphics[width=\textwidth]{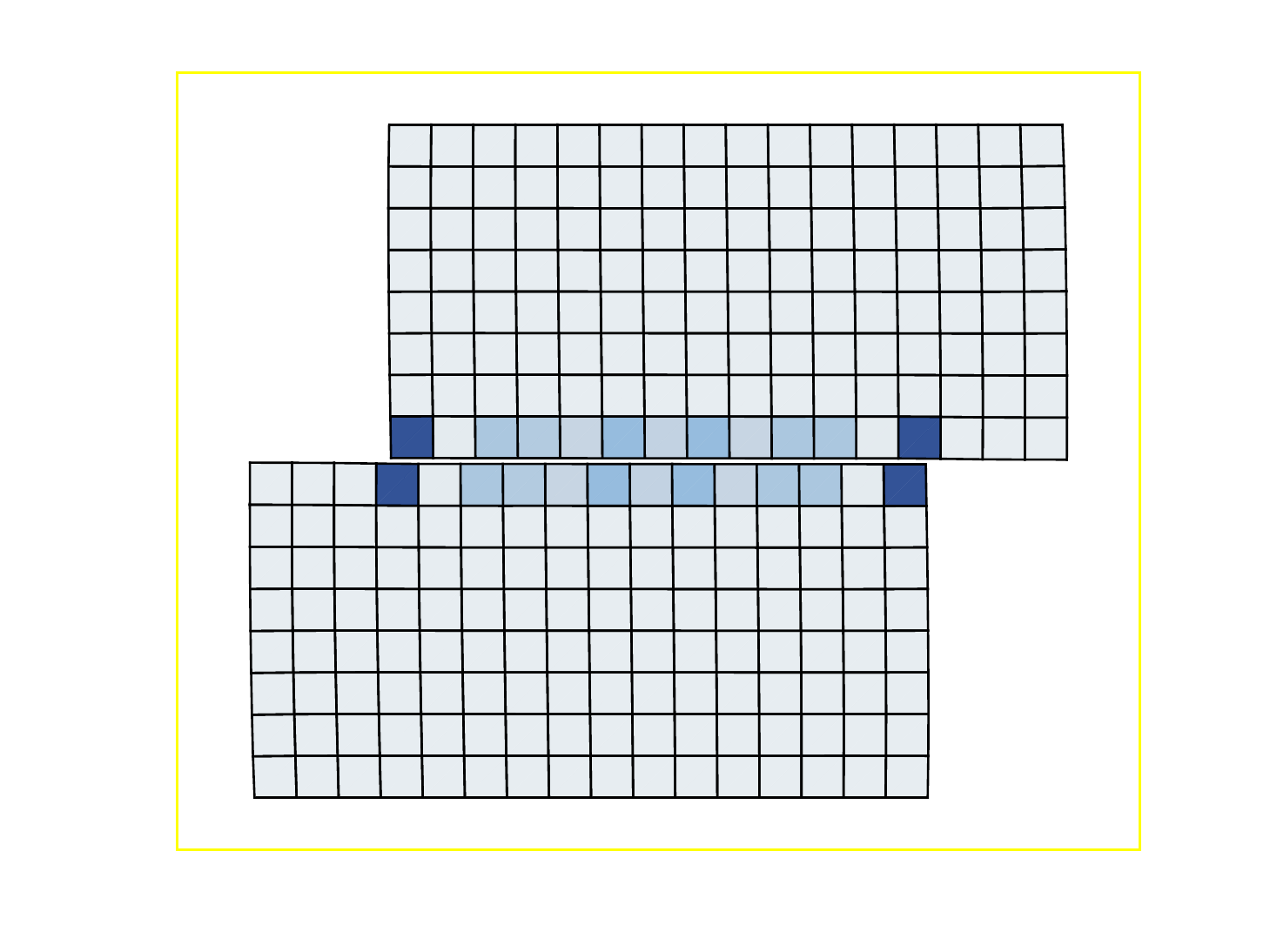}
\subcaption{$m_2=0.5, \epsilon_2=1/8$.}
\includegraphics[width=\textwidth]{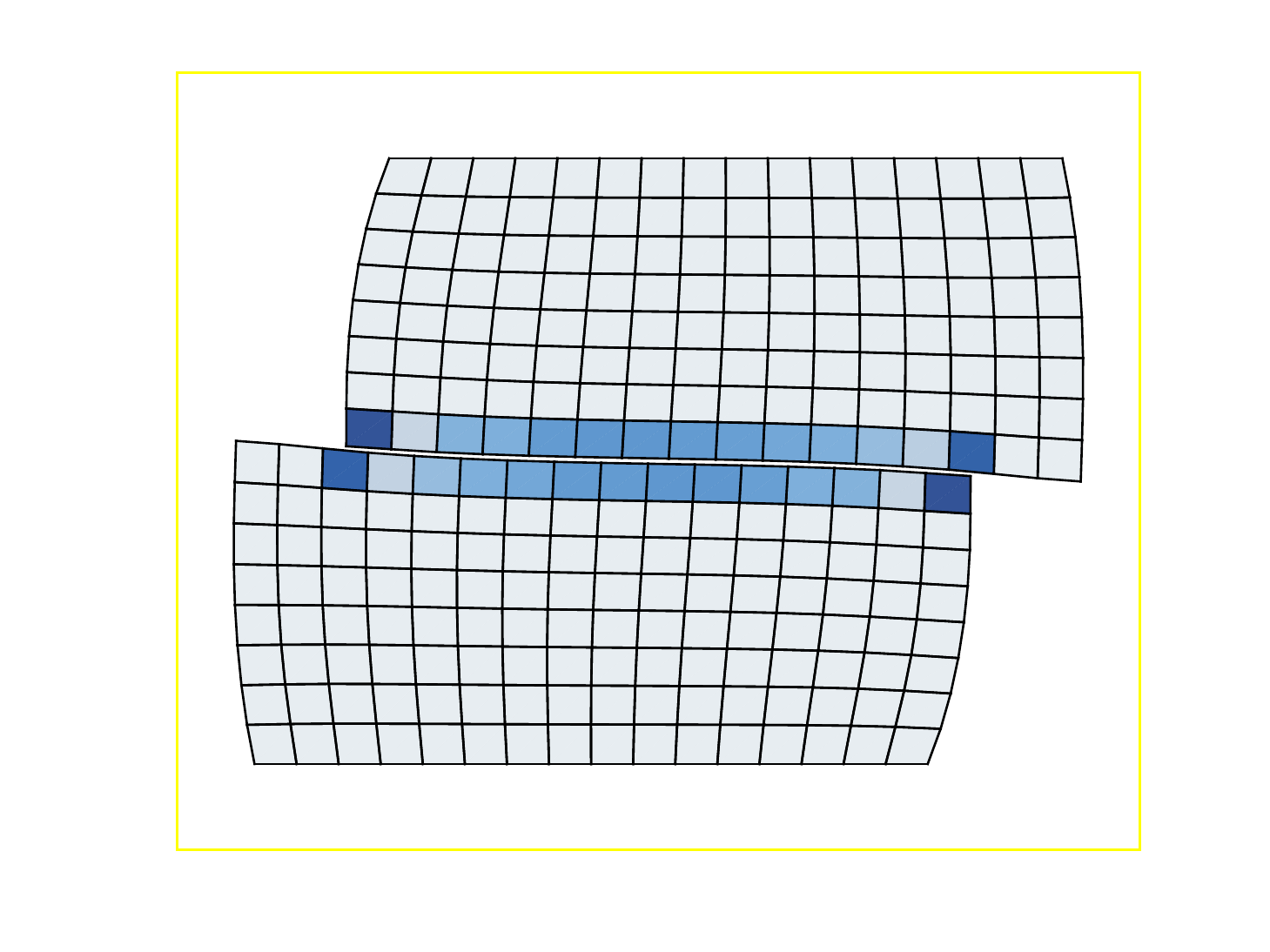}     
\subcaption{$m_2=0.6, \epsilon_2=1/8$.}		
\includegraphics[width=\textwidth]{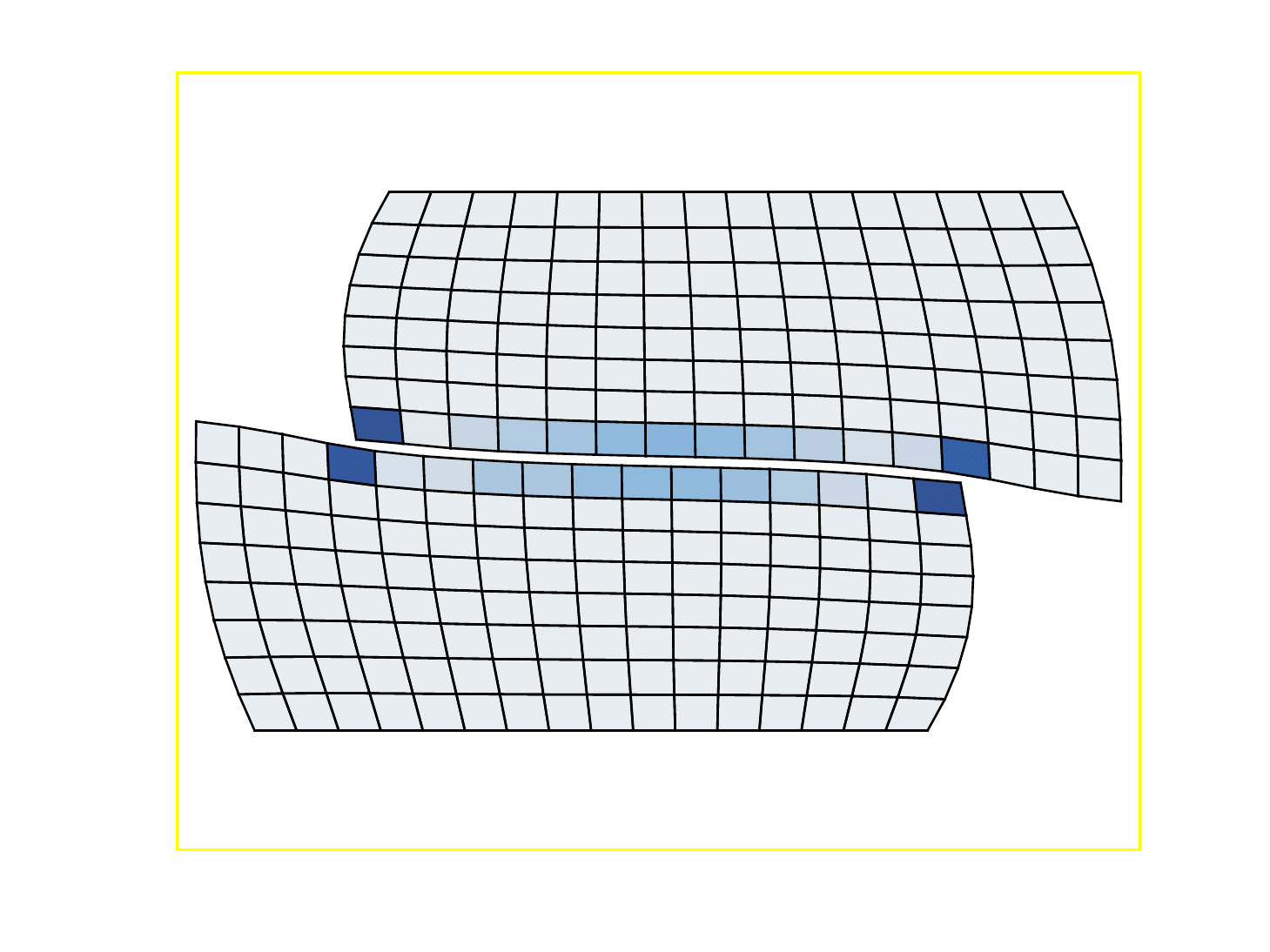}     
\subcaption{$m_2=0.7, \epsilon_2=1/8$.}
\end{minipage}  
\caption{Model I: Optimized deformed domains with underlying marginal density of $E^{CN}_{\eps_2}(y)$.}
\label{fig:model1_optimization}
\end{figure}

\begin{figure}[h]
    \begin{minipage}[t]{.8\textwidth}
        \centering
        \includegraphics[width=\textwidth]{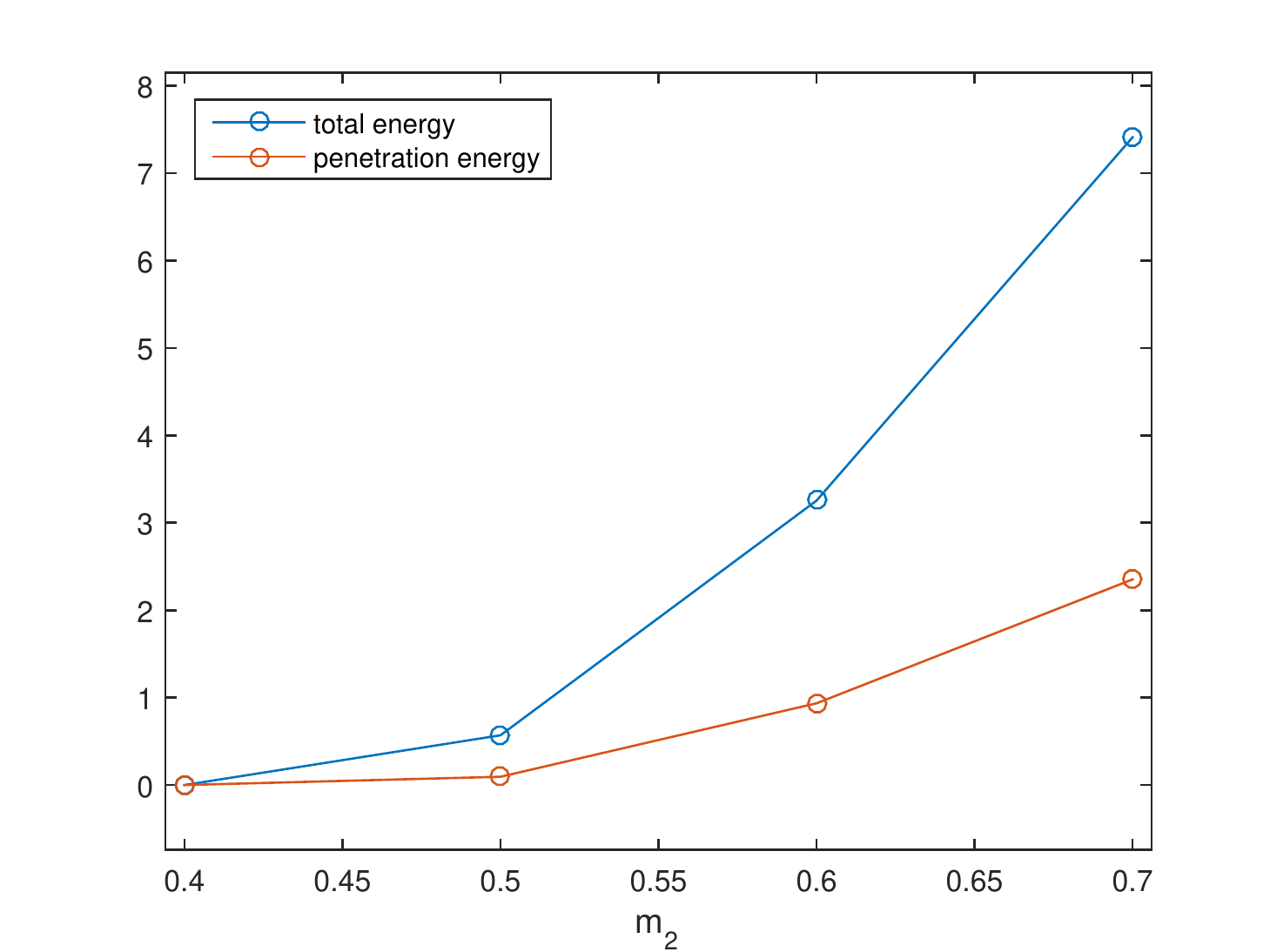}  
    \end{minipage}  
    \caption{Model I: Total energy $E_{\eps,\sigma,\mu}(y)$ and penetration energy $E^{CN}_{\eps_2}$ displayed versus the parameter $m_2$. We consider the case $\eps_2=1/4$.    }\label{fig:energies_vs_m2}
\end{figure}

\subsection{Model II}

We consider the same ``pincers'' domain $\Omega\subset \RR^2$ 
as in example of Subsection \ref{subs:example1} and subject $\Omega$ to the linear body force density
$$   g_{\rm body}(x_1, x_2) =\nu  (0, - H(x_1) \, \sign{(x_2)}) \quad \mbox{on } (x_1, x_2) \in \Omega,$$
where $H$ denotes the Heaviside step function. In addition, we impose Dirichlet boundary conditions on a part of the boundary:
$$ y_1=x_1, \quad  y_2=x_2 \quad \mbox{for } (x_1, x_2) \in \Lambda_D,$$
where $\Lambda_D := \{ 0 \} \times (-1/2, 1/2).$ See Figure \ref{fig:model1setup} for illustration.

\begin{figure}[h]
    \begin{minipage}[t]{.6\textwidth}
        \centering
				\includegraphics[width=\textwidth]{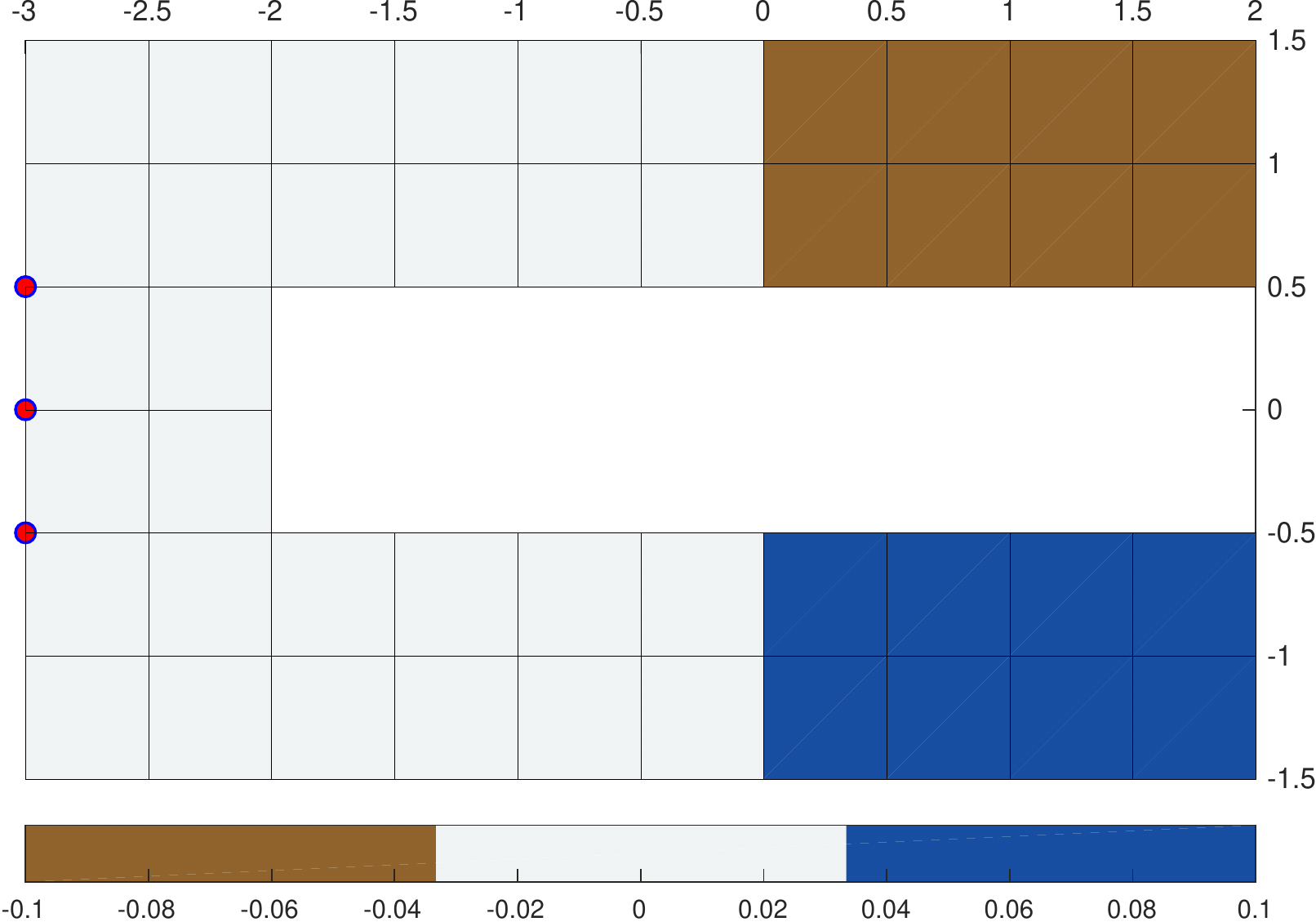}    
    \end{minipage}  
    \caption{Model II : body force density in $x_2$ direction forcing both pincers part to move against each other and Dirichlet boundary nodes indicated by full  dots.
}\label{fig:model1setup}
\end{figure}

In this model, we measure the response of an elastic continuum to various scaling of the Ciarlet-Ne\v{c}as penalty term $\mu E^{CN}_{\eps_2}$. We consider a sequence of minimization problems with various multipliers
$$ \mu \in \{10^{-6}, 
10^{-5}, 10^{-4}, 10^{-3}, 10^{-2}, 10^{-1}, 1  \} $$
 in two different cases
$$\epsilon_2=1/2, \quad \epsilon_2=1/4.$$
and the same loading given by $\nu = 0.2$. 
Figure \ref{fig:energies_vs_mu}  displays how much the total energy $E_{\eps,\sigma,\mu}(y)$ and the scaled penetration energy $\mu E^{CN}_{\eps_2}(y)$ depend on the multiplier $ \mu$. 

\begin{figure}[h]
    \begin{minipage}[t]{.9\textwidth}
        \centering
        \includegraphics[width=\textwidth]{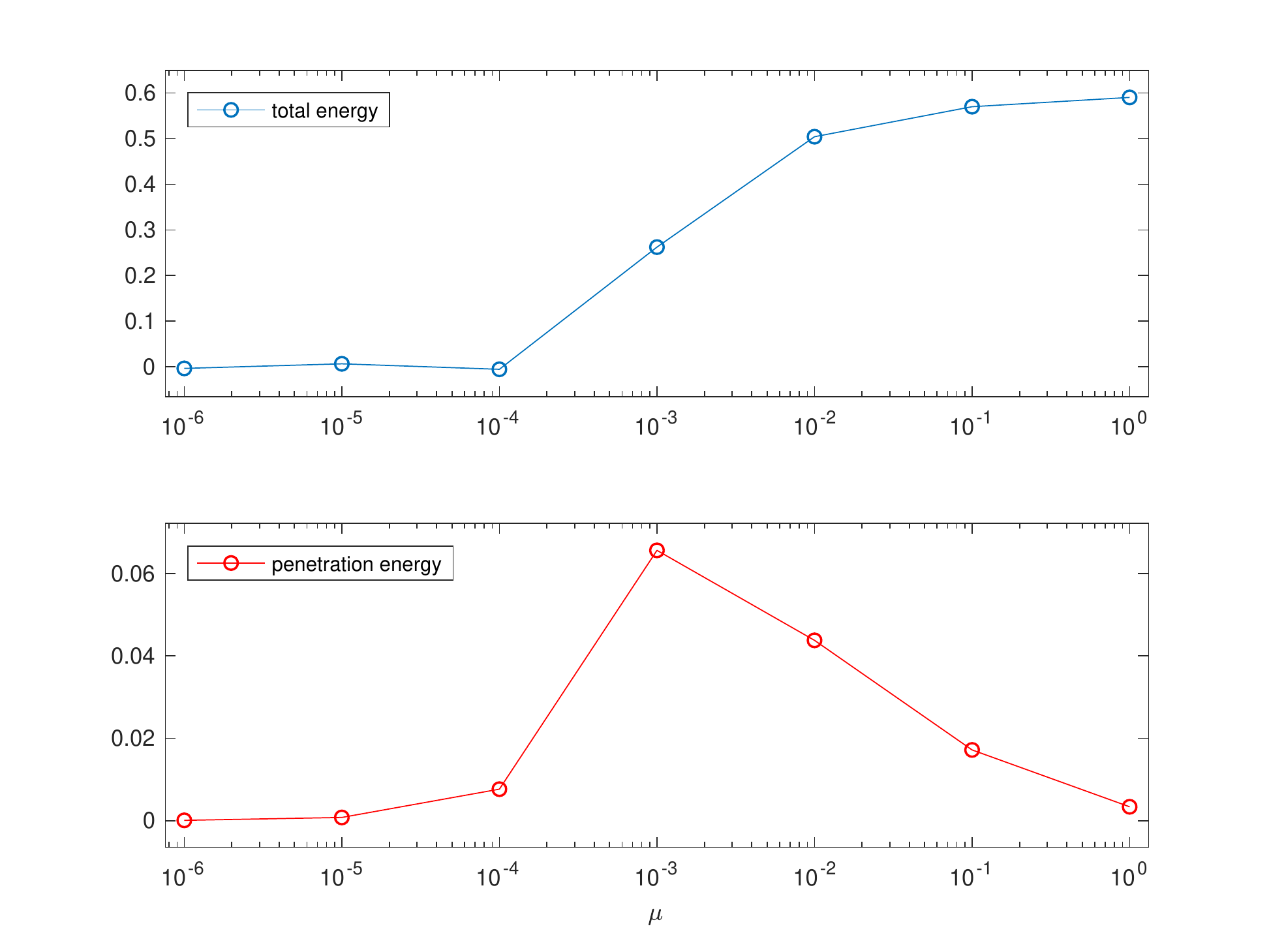} 
    \end{minipage}  
   \caption{Model II: Total energy $E_{\eps,\sigma,\mu}(y)$ and penetration energy $\mu E^{CN}_{\eps_2}$ displayed versus the parameter $\mu$. We consider the case $\eps_2=1/2$.}\label{fig:energies_vs_mu}
\end{figure}

For lower values of $\mu$ the scaled penetration term 
allows for a penetration of both pincers parts. For higher values of $\mu$ the scaled penetration term 
one can usually prevent penetration altogether. 
Here, recall that from the point of view of the theory, we only know for sure that reducing $\eps_2$ 
will eventually prevent penetration if the scaling exponent $\beta$ in $E^{CN}_{\eps_2}$ is big enough (Corollary~\ref{cor:invertibility}). Of course, at finite scales 
increasing $\mu$ has a similar effect, and as long as the grid size $h$ is fixed, we cannot arbitrarily reduce $\eps_2$.

\begin{figure}[H]
\begin{minipage}[t]{.45\textwidth}
\centering
\includegraphics[width=\textwidth]{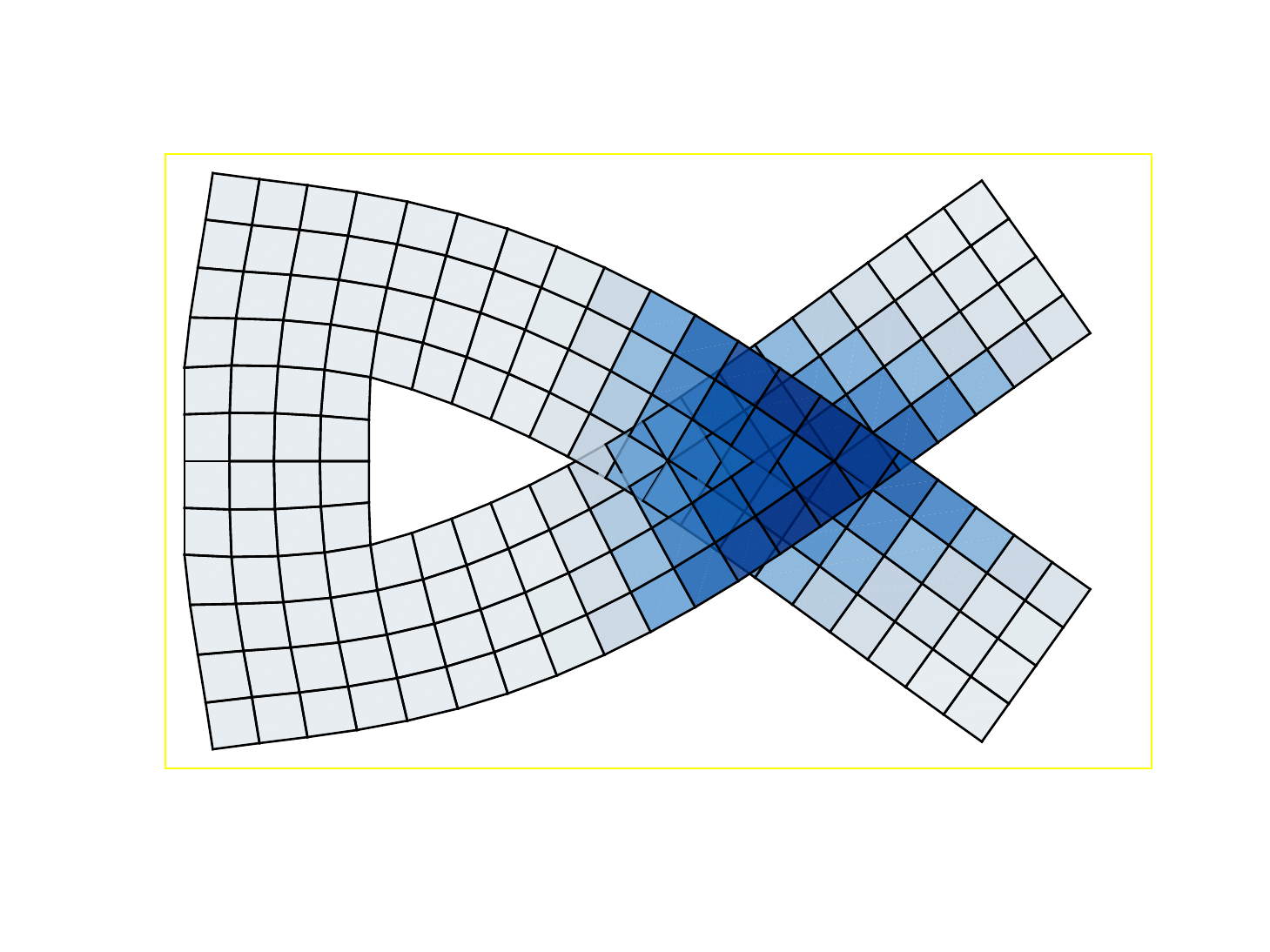} 
\subcaption{$\mu_2=10^{-4}, \epsilon_2=1/2$.}
\includegraphics[width=\textwidth]{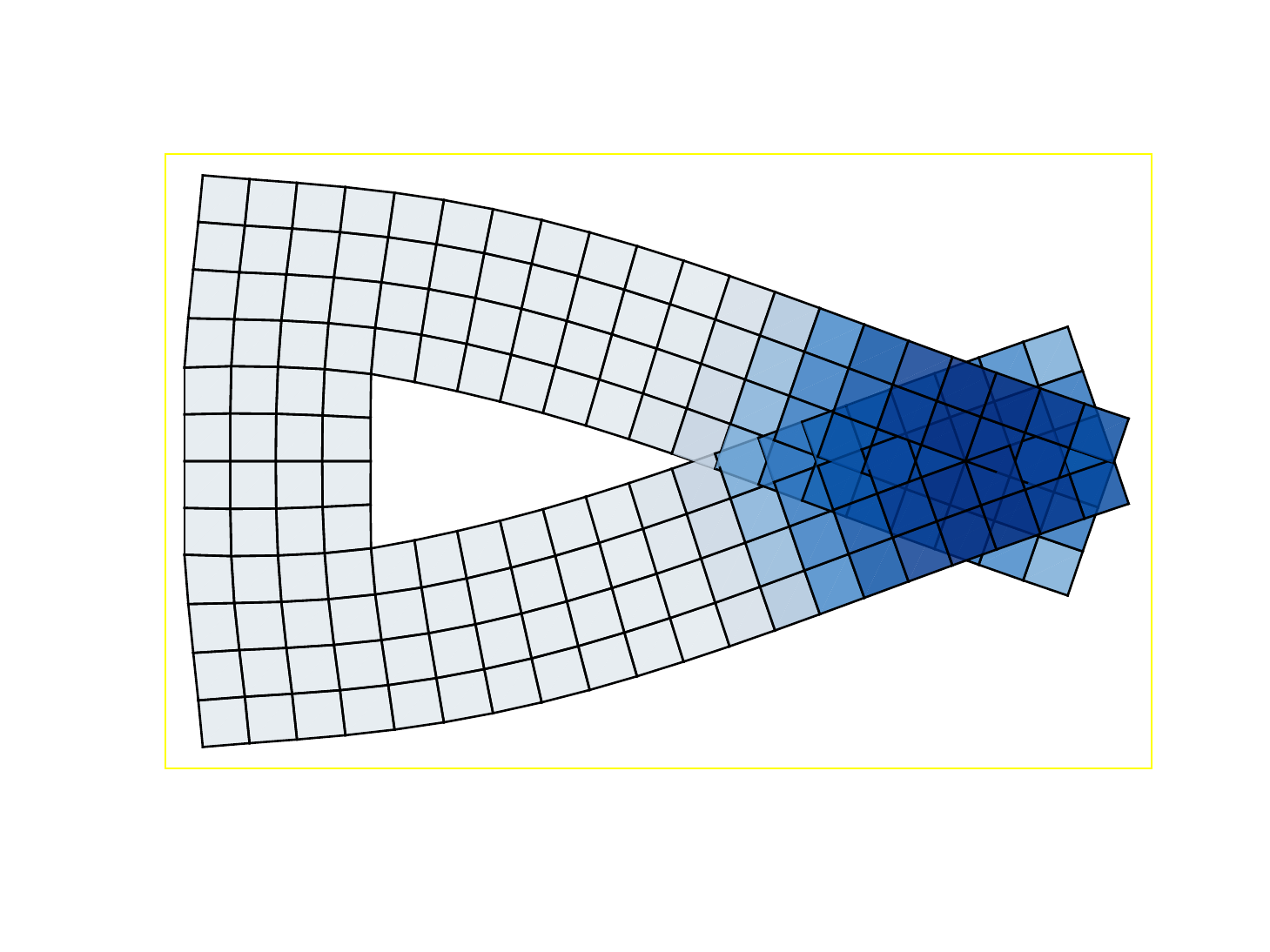} 			
\subcaption{$\mu_2=10^{-3},\epsilon_2=1/2$.}
\includegraphics[width=\textwidth]{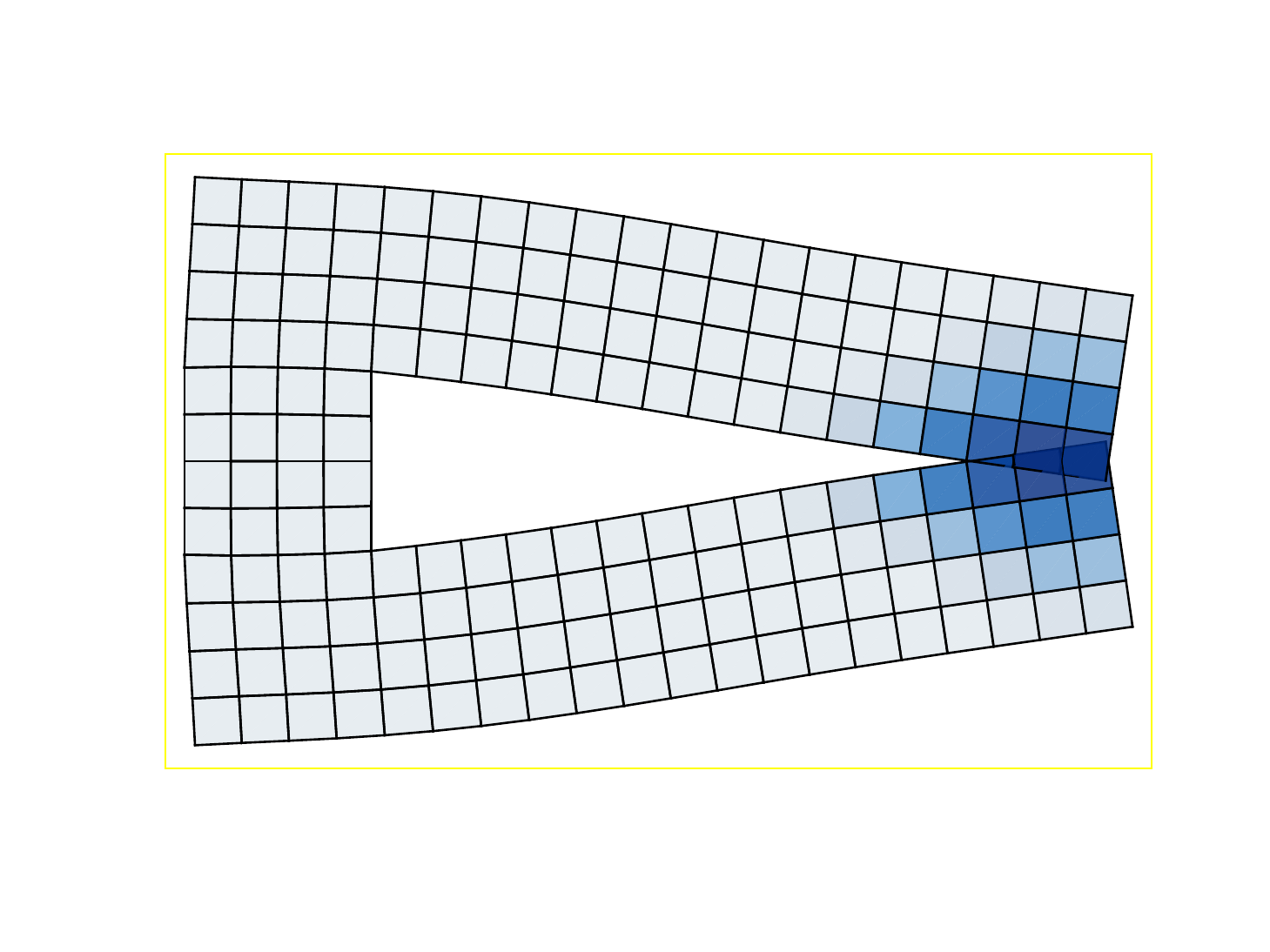} 
\subcaption{$\mu_2=10^{-2},\epsilon_2=1/2$.}
\includegraphics[width=\textwidth]{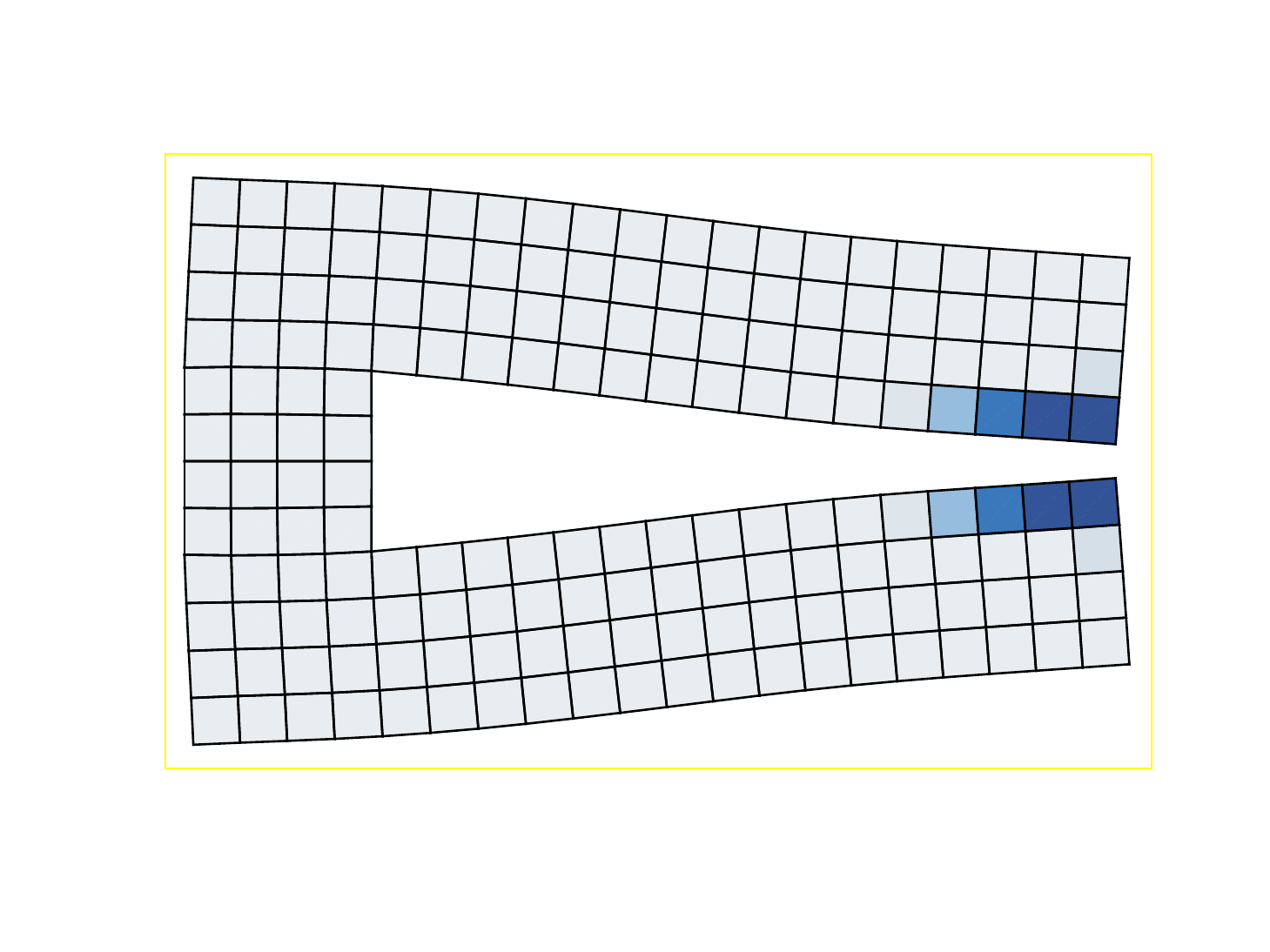} 
\subcaption{$\mu_2=10^{-1},\epsilon_2=1/2$.}
\end{minipage}
\hfill
\begin{minipage}[t]{.45\textwidth}
\centering
\includegraphics[width=\textwidth]{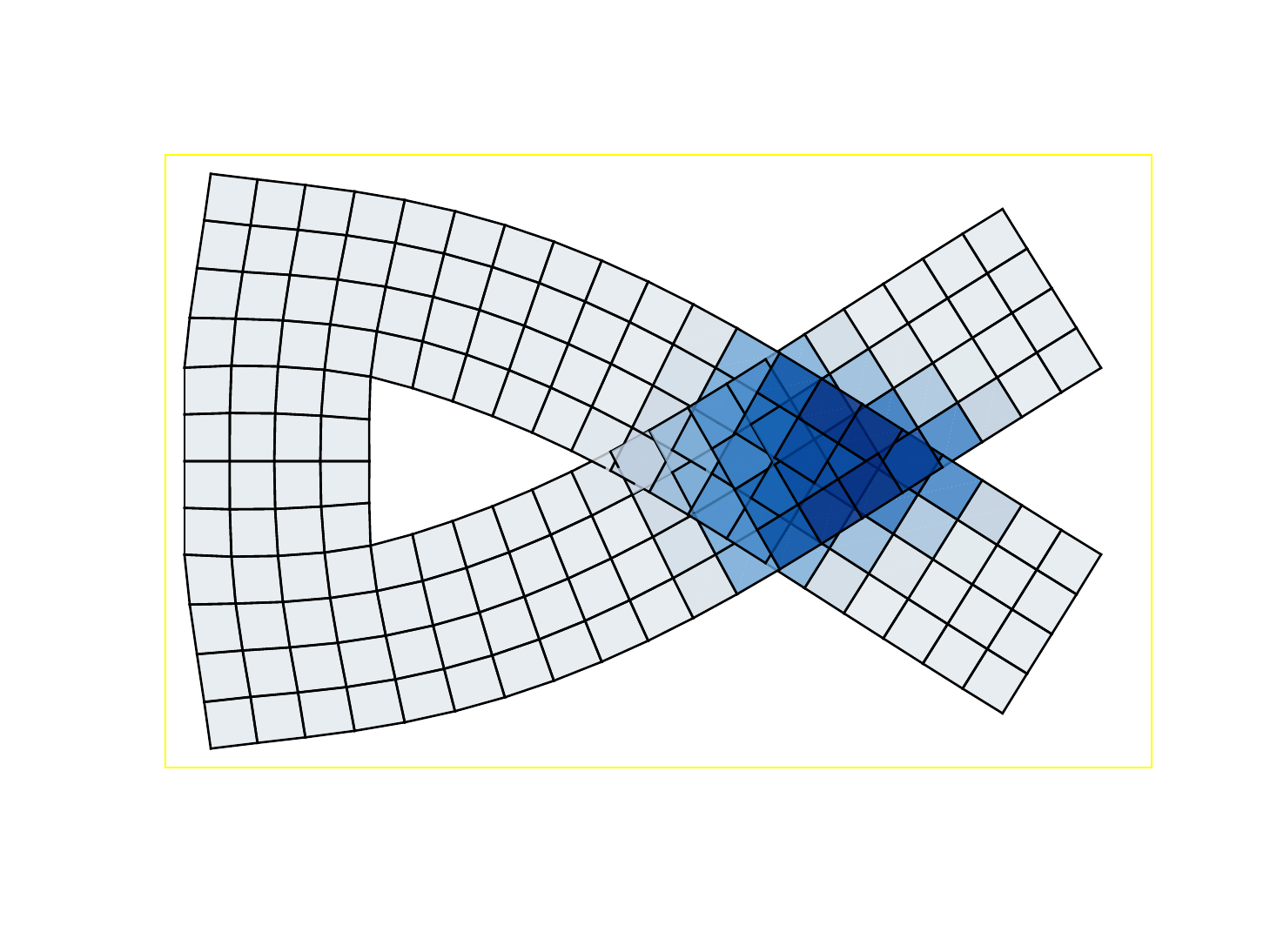} 
\subcaption{$\mu_2=10^{-4}, \epsilon_2=1/4$.}
\includegraphics[width=\textwidth]{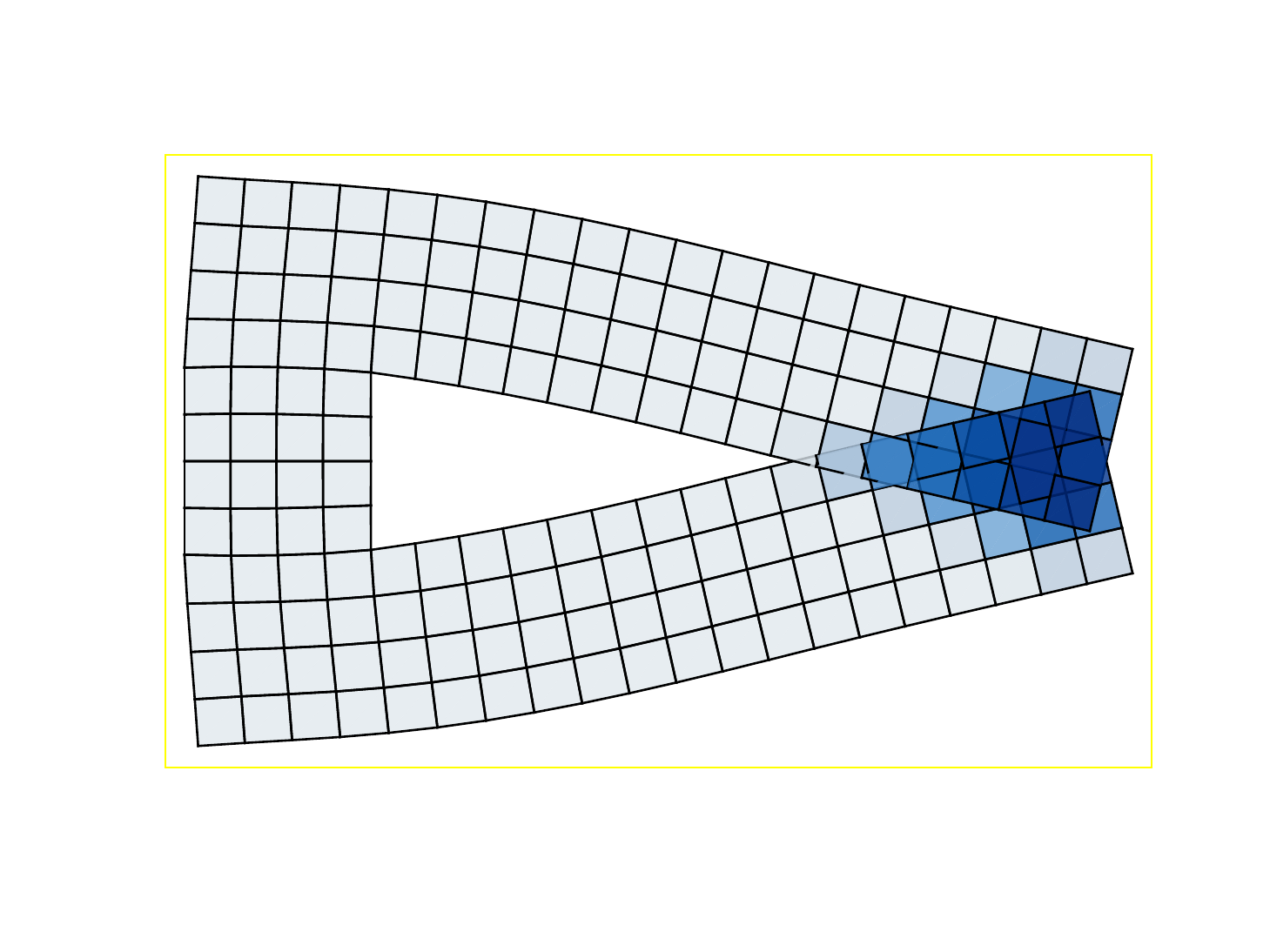} 			
\subcaption{$\mu_2=10^{-3},\epsilon_2=1/4$.}
\includegraphics[width=\textwidth]{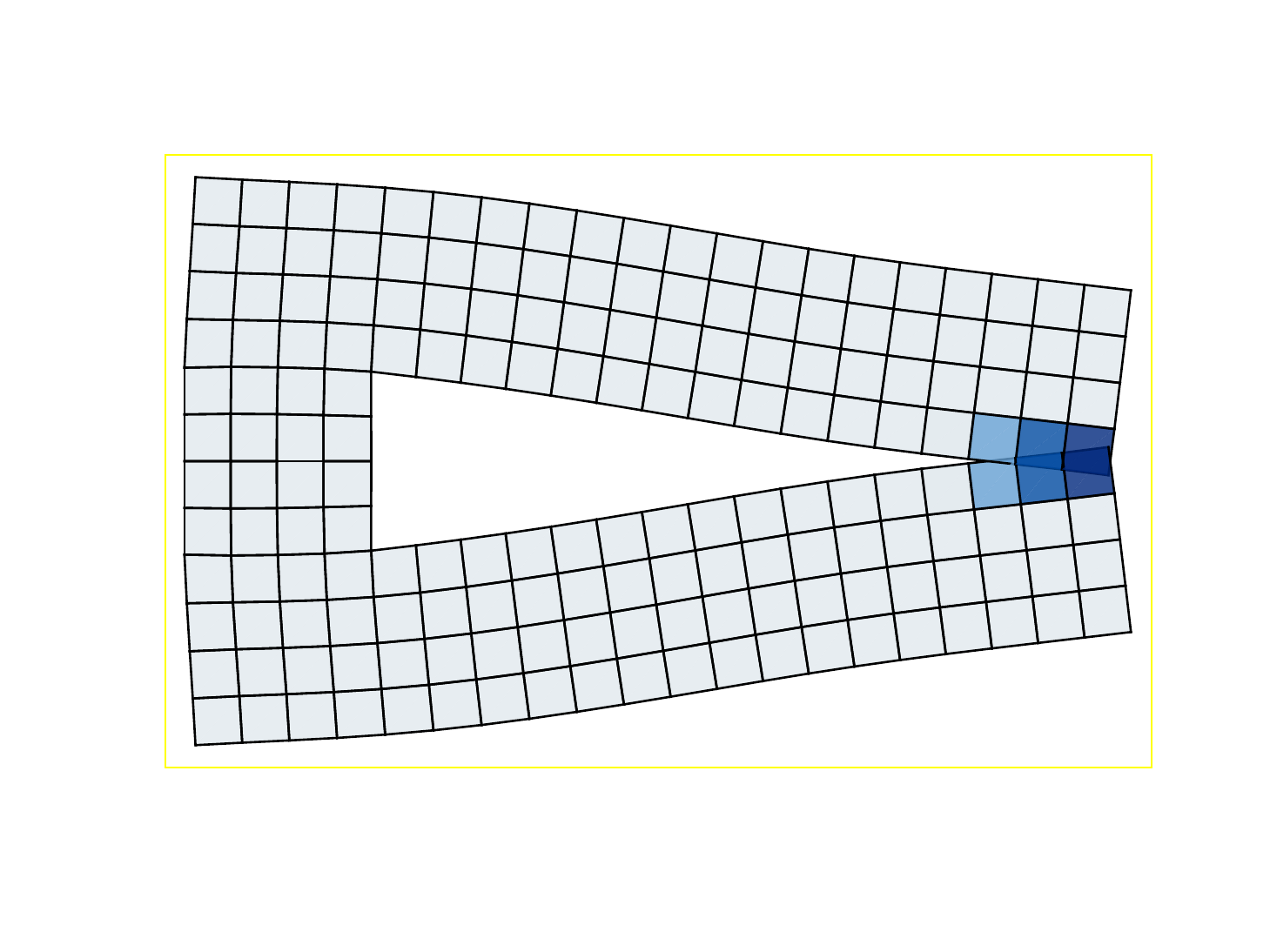} 
\subcaption{$\mu_2=10^{-2},\epsilon_2=1/4$.}
\includegraphics[width=\textwidth]{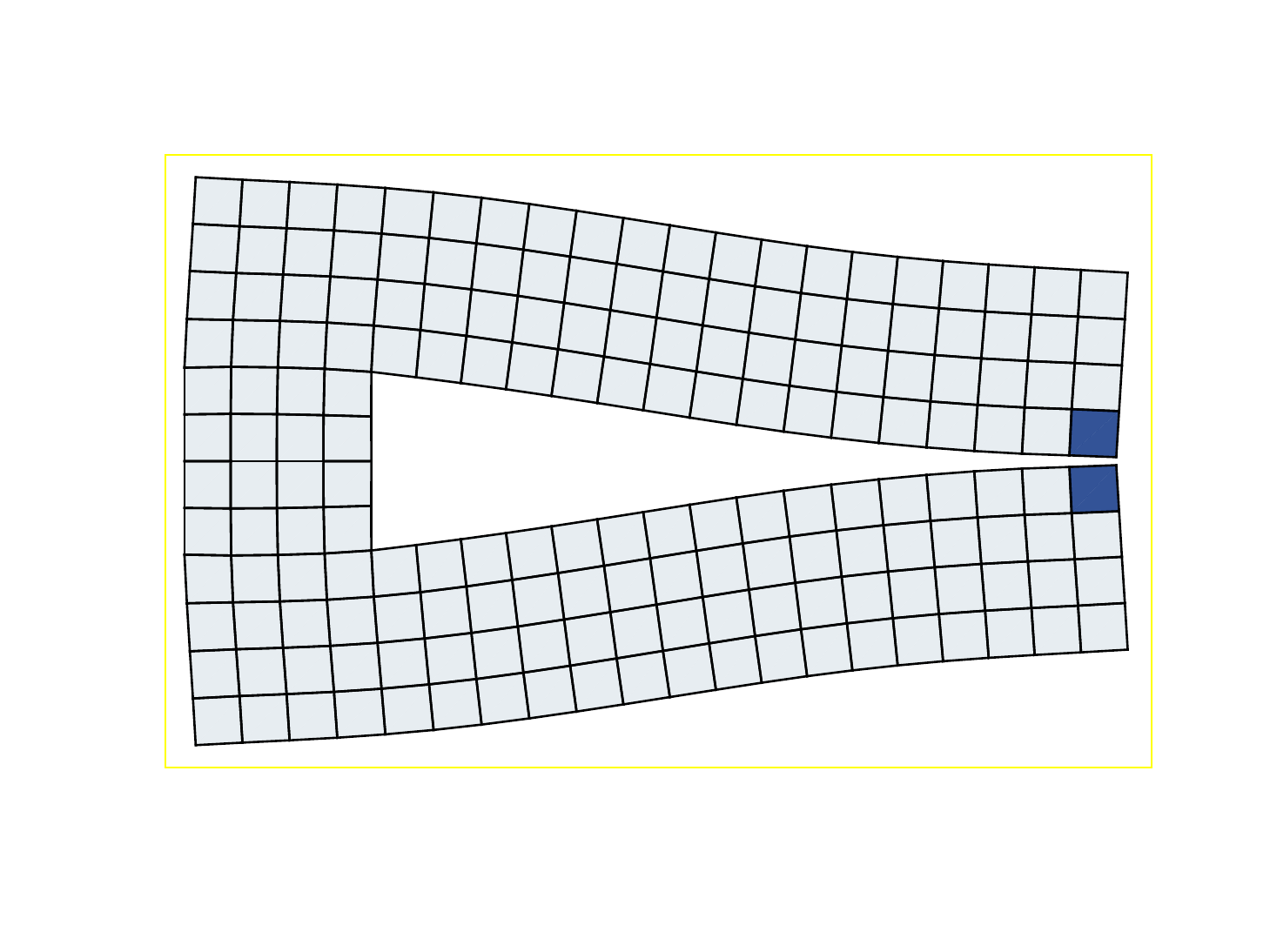} 
\subcaption{$\mu_2=10^{-1},\epsilon_2=1/4$.}
\end{minipage}  
\caption{Model II: Optimized deformed domains with underlying marginal density of $\mu E^{CN}_{\eps_2}(y)$.}
\label{fig:model2_optimization}
\end{figure}

\subsection{Remarks on implementation}

Our Matlab code is based on former codes related to \cite{AnVa15,HaVa14, RaVa13} that allow for a vectorized assembly of finite element matrices. 
The code available for download and includes an own implementation of the Bogner-Fox-Schmit (BFS) rectangular elements for a uniformly refined rectangular mesh, where all rectangular elements are for simplicity of the same size $hx_1 \times hx_2$. Both Model I and Model II rectangular meshes are of this kind. The basis functions on each rectangle are based on bicubic polynomials, i.e. tensor products of 4 cubic (Hermite) polynomials. They have 16 degrees of freedom with 4 degrees in each of its 4 corner nodes approximating: a function value, its gradient and the second mixed derivative. Therefore, a given scalar function 
$u \in C^1(\Omega)$ is represented by a matrix of the size $nn \times 4$ in the form
$$
\underline{\bf{u}}=
\begin{pmatrix}
u(x_1^1, x_2^1) & \frac{\partial u}{\partial x_1}(x_1^1, x_2^1) & \frac{\partial u}{\partial x_2}((x_1^1, x_2^1)& \frac{\partial^2 u}{\partial x_1 \partial x_2}(x_1^1, x_2^1) \\
\dots & \dots & \dots & \dots \\
u(x_1^{nn}, x_2^{nn}) & \frac{\partial u}{\partial x_1}(x_1^{nn}, x_2^{nn}) & \frac{\partial u}{\partial x_2}((x_1^{nn}, x_2^{nn})& \frac{\partial^2 u}{\partial x_1 \partial x_2}(x_1^{nn}, x_2^{nn})
\end{pmatrix},
$$
where $nn$ denotes the total number of mesh nodes and $(x_1^{i}, x_2^{i})$ for $i=1, \dots, nn$ their corresponding coordinates. The construction of BFS elements additionally guarantees $\frac{\partial^2 u}{\partial x_1 \partial x_2} \in C(\Omega)$ but the remaining second-order derivatives are generally discontinuous. Based on a global numbering of nodes, the matrix $\underline{\bf{u}}$ is further reformated as a column vector $\bf{u}$ with $4 \cdot nn$ entries. For our 2d nonlinear elasticity computations, we approximate both components $y_1, y_2$ by the BFS elements and resulting vector variable $y=(y_1, y_2) \in C^1(\Omega;\RR^2)$ has $8 \cdot nn$ entries.


\begin{figure}
    \begin{minipage}[c]{.28\textwidth}
        \centering
        \includegraphics[width=\textwidth]{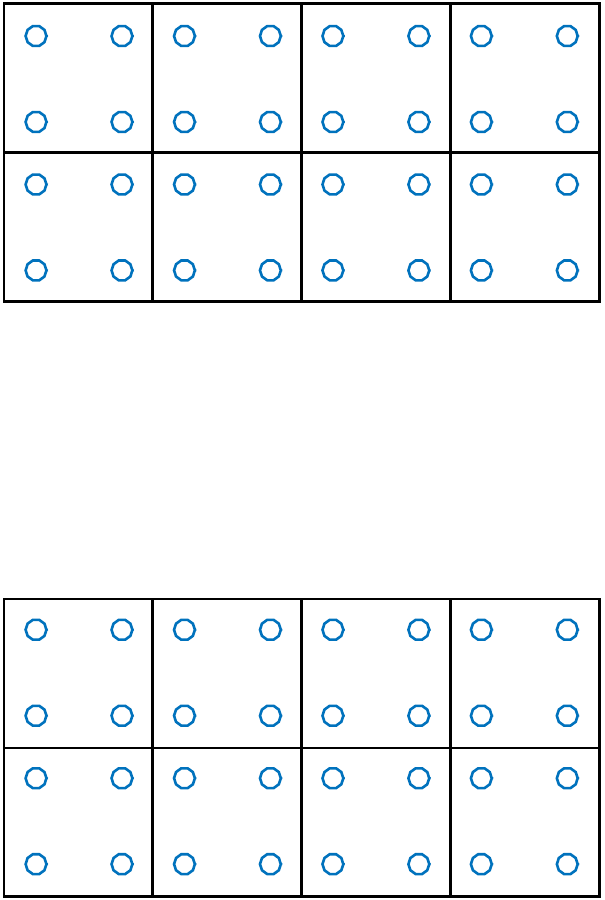}
    \end{minipage}  
		\hspace{0.09\textwidth}
		\begin{minipage}[c]{.28\textwidth}
        \centering
        \includegraphics[width=\textwidth]{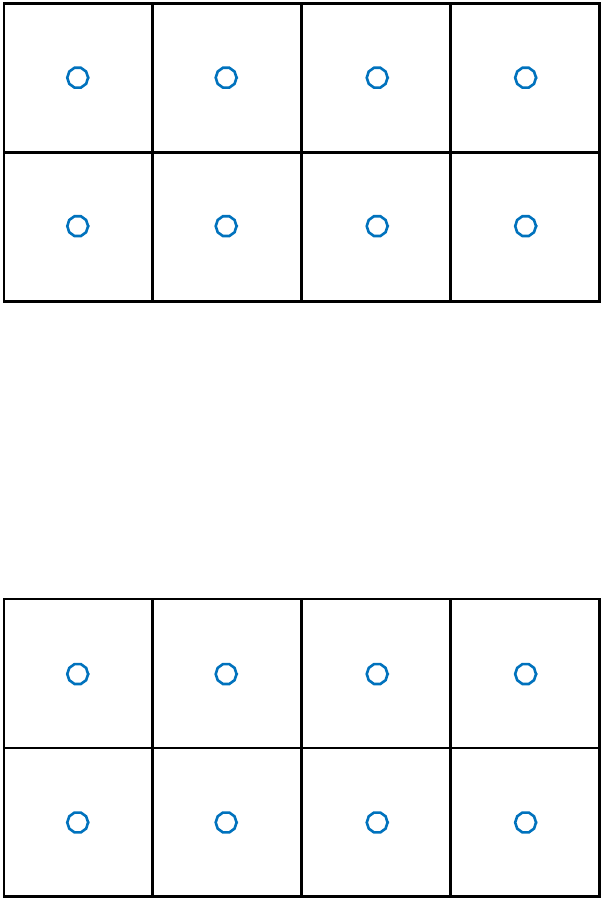}
    \end{minipage}  
		 \caption{Model I : Four Gauss integration points (left) used for evaluation of all energy parts except the penalty term, midpoints (right) used in the evaluation of the penetration penalty term.}
   \label{fig:model1_ips}
\end{figure}

Since energy parts of $E_{\eps,\sigma}$ are generally non-quadratic functionals, all two-dimensional integrals are evaluated using the Gaussian quadrature, where integration points are tensor products of one-dimensional Gauss integration points. We deploy four Gauss quadrature points  in each rectangle as illustrated in Figure \ref{fig:model1_ips} (left). 

The penetration penalty term $E^{CN}_{\eps_2}$ is a nonlocal functional. We make no additional assumptions on the location of the penetration and evaluate all pairwise euclidean distances 
$$|(x_1^{i}, x_2^{i}) -  (x_1^{j}, x_2^{j})  |, \qquad |(y_1^{i}, y_2^{i}) -  (y_1^{j}, y_2^{j})  |$$
in a double loop over $i, j = 1,\dots, ne$. Here, vectors above denote coordinates of rectangles midpoints (cf. the right part of Figure \ref{fig:model1_ips}) and their corresponding deformations and $ne$ the number of mesh rectangles. The x-distances above are precomputed, the y-distances need to be recomputed in every evaluation of the penetration penalty term. 

\begin{rem}[Possible implementation improvement]
Even if no assumptions are made on the location of the penetration,
it is possible (but not implemented here) to optimize the evaluation of $E^{CN}_{\eps_2}$ as follows:
\begin{itemize}
\item Instead of a full double loop, first only go through all pairs of elements located at the boundary of the domain. Create a list of those elements that contribute to $E^{CN}_{\eps_2}$.
\item Then start to search for other contributing elements in the interior by repeatedly checking all elements that are neighbors of those that are already known to give a positive contribution.
\item Stop when no new contributing neighbor elements are found. 
\end{itemize}
While the full double loop requires a number of steps of the order of $h^{-2d}$ for the mesh size $h$, the double loop through the elements at the boundary only needs $h^{-2(d-1)}$. As long as there is no deep penetration (penetration depth of the order of $h$ or less) and $\eps_2=O(h)$ ($\eps_2\geq h$, but not that much bigger), a subsequent search for contributing neighbor elements does not increase that significantly, either.
\end{rem}

\section*{Acknowledgements}
Our research was supported by the Czech Science Foundation (GA \v{C}R), through the grants GA18-03834S (SK and JV) and GA17-04301S (JV).


\bibliographystyle{plain}
\bibliography{NLEbib}

\end{document}